\documentclass[12pt]{article}

\usepackage{inputenc}%[latin1]
\usepackage{amsmath}
\usepackage{amsfonts}
\usepackage{amssymb}
\usepackage[T1]{fontenc}
\usepackage{graphicx}
\usepackage[english]{babel}
\usepackage{mathrsfs}
\usepackage{amsthm}
\usepackage[all]{xy}
\usepackage{geometry}
\allowdisplaybreaks
\usepackage{color}

\title{ Properties of steady states for a class of non-local Fisher-KPP equations in disconnected domains }
\author{Alexis L\'{e}culier\footnote{Institut de Math\'{e}matiques de Toulouse; UMR 5219, Universit\'{e} de Toulouse; CNRS, UPS IMT, F-31062 Toulouse Cedex 9, France; E-mail:  Alexis.Leculier@math.univ-toulouse.fr}  and Jean-Michel Roquejoffre\footnote{Institut de Math\'{e}matiques de Toulouse; UMR 5219, Universit\'{e} de Toulouse; CNRS, UPS IMT, F-31062 Toulouse Cedex 9, France; E-mail:  Jean-Michel.Roquejoffre@math.univ-toulouse.fr} }

\begin{document}                     
\maketitle

\begin{abstract}

%\noindent \textbf{Abstract } 
The question studied here is the existence and uniqueness of a non-trivial bounded steady state of a Fisher-KPP equation involving a fractional Laplacian $(-\Delta)^\alpha$ in a domain with Dirichlet conditions outside of the domain. More specifically, we investigate such questions in the case of general fragmented unbounded domains. Indeed, we take advantage of the non-local dispersion in order to provide analytic bounds (which depend only on the domain) on the steady states. Such results are relevant in biology. For instance, our results provide criteria on the domain for the subsistence of a species subject to a non-local diffusion in a fragmented area. These criteria primarily involve the sign of the first eigenvalue of the operator $(-\Delta)^\alpha - Id$ in a domain with Dirichlet conditions outside of the domain. To this end, we exhibit a result of continuity of this principal eigenvalue with respect to the distance between two compact patchs in the one dimensional case. The main novelty of this last result is the continuity up to the distance $0$.

\end{abstract}

\tableofcontents
%\newpage

\newtheorem{theorem}{Theorem}
\newtheorem{corollary}{Corollary}
\newtheorem{lemma}{Lemma}
\newtheorem{definition}{Definition}
\newtheorem{proposition}{Proposition}
\newtheorem{definition-proposition}{Definition-Proposition}
\newtheorem*{notation}{Notation}
\newtheorem*{remark}{Remark}
\newtheorem{example}{Example}

\section{Introduction}

\subsection{Model, question, motivation}
We  investigate here  existence and uniqueness of bounded positive solutions for the fractional Fisher-KPP equations of the form
\begin{equation}\label{stationary_state}
\text{ i.e. }\left\lbrace
\begin{aligned}
(-\Delta)^\alpha n_+(x) &= n_+(x)-n_+^2 (x) && \ \text{ for } x \in \Omega, \\
n_+(x) &= 0 											&& \ \text{ for } x \in\Omega^c.
\end{aligned}
\right.
\end{equation}
The domain $\Omega$ is an infinite union of patches, all of them but perhaps one being bounded.  The operator $(-\Delta)^\alpha$ is the fractional Laplacian:
\begin{equation}(-\Delta)^\alpha \phi (x)= C_{\alpha}\ PV \ \int_{\mathbb{R}^d} \frac{\phi(x)- \phi(y)}{|x-y|^{d+2\alpha}}dy \qquad   \text{ with } \  C_\alpha = \frac{4^\alpha \Gamma(\frac{d}{2} +\alpha )}{\pi^{\frac{d}{2}}|\Gamma(-\alpha)|},
\end{equation}
%The existence of a steady state of \eqref{stationary_state} involves the previous continuity result and the uniqueness is proved thanks to accurate estimates from below and above of any bounded steady states $n_+$.  
we assume $\alpha<1$ throughout the work. Clearly, existence and uniqueness would be false if $\alpha=1$ (just think of a periodic union of large line segments), but nonlocality implies a solidarity between patches that may make existence and uniqueness become true. We will first see some effects of the nonlocality when dealing with principal eigenvalue problems, we will then try to understand how the solidarity forced by nonlocal diffusion eventually leads to existence and uniqueness. More important in our opinion, we will take advantage of the disconnectedness of the domain to derive precise estimates of possible solutions of \eqref{stationary_state} at infinity, that will eventually imply uniqueness.
 
 \noindent The evolution equation
\begin{equation}\label{equation_with_time}
\left\lbrace
\begin{aligned}
\partial_t n (x,t) + (-\Delta)^\alpha n(x,t) &= n(x,t)-n(x,t)^2 && \ \text{ for } (x,t) \in \Omega \times ]0,+\infty[,\\
n(x,t) &= 0 											&& \ \text{ for } (x,t) \in \Omega^c \times [0, +\infty[,\\
n(x,0) &= n_0(x),
\end{aligned}
\right.
\end{equation}
models biological invasions. The variable $n$ stands for a density of population. The fractional Laplacian  models the fact that a species can jump from one point to another with a high rate. If a bounded solution of \eqref{equation_with_time} converges as $t$ tends to $+\infty$, it is either to $0$, either to a non-trivial stationary state of \eqref{equation_with_time} : a solution of \eqref{stationary_state}. Thus, if \eqref{stationary_state} does not admit a bounded positive solution, we deduce that the species modeled by $n$ will extinct. On the other hand, if there exists a unique bounded positive solution $n_+$ to \eqref{stationary_state} to which the solution $n$ converges to $n_+$ as $t$ tends to $+\infty$, the species will persist.  An equation of type \eqref{equation_with_time} was first introduced in 1937 by Fisher in \cite{Fisher} and by Kolmogorov, Petrovskii and Piscunov in \cite{KPP1937} in the whole domain $\mathbb{R}^d$ and with a standard diffusion. 

\noindent Our model accounts for a situation where reproduction is allowed on some patches (perhaps an infinity of them) while the outside environment is lethal to the species. Moreover, the patches may be, individually, unfavourable to reproduction, and this may be true for ll but one of them. The question is whether the species may survive in these conditions, and if such is the case, in which quantity. Of course this can only be possible because of nonlocal diffusion, survival being clearly impossible in the conditions just described, if the diffusion is the usual one. One of the main result of this work is that there is survival but if all the patches but one are individually unfavourable, the density of individuals will decay like a power of the distance to the favourable patch that we evaluate precisely.

\noindent Before that, we account for specfic effects of the nonlocal diffusion for simple one-dimensional domains.

\subsection{Some effects of nonlocality on the principal eigenvalue}\label{evsection}
We consider a domain made up of two patches of variable distance, and we wonder  how the principal eigenvalue is affected by the distance between the patches. In  particular, we ask whether it is  continuous with respect to this distance. This question is of course irrelevant if the fractional Laplacian is replaced by the usual Laplacian, as both domains have a principal eigenvalue of their own. However, in the nonlocal case, there is a solidarity between the patches and the question becomes relevant. In the one dimensional case, we give a positive answer especially when the distance tends to $0$. 
Our domain is of the form
\begin{equation}\label{Omgega12mu}
    \Omega_{1,2,\mu}=\Omega_1 \cup \Omega_2 \quad  \text{ and } \quad  \mathrm{dist}
(\Omega_1, \Omega_2) = \inf \left\lbrace |x_1- x_2| \text{ with } x_1 \in \Omega_1, \ x_2 \in \Omega_2 \right\rbrace = 2 \mu.
\end{equation}
\begin{notation}
For any smooth bounded set $\mathcal{O}$, let $\xi_\alpha ( \mathcal{O})$ be the principal eigenvalue of the fractional Laplacian with an exponent $\alpha$ with Dirichlet conditions outside the domain $\mathcal{O}$
\begin{equation}
\text{i.e.} \quad \left\lbrace 
\begin{aligned}
& (-\Delta)^\alpha \phi_\alpha = \xi_\alpha(\mathcal{O})\phi_\alpha && \text{ in } \mathcal{O}, \\
& \phi_\alpha  = 0 && \text{ in  }\mathcal{O}^c.
\end{aligned}
\right.
\end{equation}
\end{notation}
\noindent The existence of such eigenvalue is ensured by the Krein-Rutman Theorem. %(see \cite{Smoller}).  We will first focus on the continuity of $\xi_\alpha(\Omega_{1,2,\mu})$ by varying the distance $\mu$ and the size of $\Omega_i$. Next, we study the continuity with respect the exponent $\alpha$. Finally, we close this discussion by varying both $\mu$ and $\alpha$. 
\noindent In this part, we will adopt the following notation:
\[\underline{\xi}_{i,\alpha} = \xi_\alpha (\Omega_i). \]
We denote by $\phi_{1,2,\mu, \alpha}$ and $\underline{\phi}_{i,\alpha}$ the eigenfunctions associated respectively to $\xi_{\alpha}(\Omega_{1,2,\mu})$ and $\underline{\xi}_{i,\alpha}$.
\subsubsection{The main continuity result}
As previously said, a two-piece domain has a first principal eigenvalue, and that this eigenvalue is continuous under the mutual distance of the two pieces is intuitively obvious, as soon as they remain far apart. When they are put together, continuity still holds: this result has of course no equivalent in the case of the standard Laplacian. Here is the precise statement.
\begin{theorem}\label{theorem_3}
Under the previous assumptions, the function $(\mu \in ]0, +\infty[  \mapsto \xi_\alpha (\Omega_{1,2,\mu }) )$ is increasing and continuous. Moreover, it is continuous up to $\mu=0$ and 
\[ \xi_{\alpha }( \Omega_{1,2,0} ) = \xi_\alpha ( ]0, |\Omega_1| + |\Omega_2|[ ) .\]
\end{theorem}

A first ingredient for the proof of Theorems \ref{theorem_3} and the monotonicity of $(\mu \mapsto \lambda_\alpha (\Omega_{1,2,\mu }))$ is  the Rayleigh quotient
\begin{equation}\label{def_R}
R_\alpha (\phi) = \frac{ \int_{\mathbb{R}} (-\Delta)^\alpha \phi(x) \phi(x)  }{\int_\mathbb{R} \phi^2(x)dx } = \frac{ \int_{\mathbb{R}} \frac{1}{2} \int_\mathbb{R} \frac{(\phi(x) - \phi(y))^2}{|x-y|^{1+2\alpha}}dy  }{\int_\mathbb{R} \phi^2(x)dx } := \frac{[\phi]_\alpha }{\|\phi \|^2_2}.
\end{equation}
%where $[\phi]_\alpha  = \frac{1}{2} \int_\mathbb{R} \int_\mathbb{R} \frac{(\phi(x)-\phi(y))^2}{|x-y|^{1+2\alpha}}dy dx$.\\
We will evaluate $R_\alpha$ in the following spaces, for a general set $\mathcal{O}$: 
\begin{equation} \label{def_H_alpha_0}
H^\alpha_0 ( \mathcal{O} ) = \left\lbrace \phi \in L^{2}(\mathcal{O}) \text{ such that }[\phi]_\alpha < +\infty\text{ and } \phi_{|\mathbb{R} \backslash \mathcal{O}} = 0 \right\rbrace.
\end{equation}
The link between the Rayleigh quotient and the principal eigenvalues is the following : 
\begin{equation}
\xi_\alpha (\Omega_{1,2,\mu }) = \underset{\phi \in H^\alpha_0(\Omega_{1,2,\mu})\backslash \left\lbrace 0 \right\rbrace }{\min} R_\alpha(\phi) = R_\alpha (\phi_{1,2,\mu,\alpha})  \   \text{ and }  \   \underline{\xi}_{i,\alpha} = \underset{\phi \in H^\alpha_0(\Omega_i)\backslash \left\lbrace 0 \right\rbrace }{\min} R_\alpha(\phi) = R_\alpha (\underline{\phi}_{i,\alpha}).
\end{equation}
The proof of the continuity of $\xi_\alpha (\Omega_{1,2,\mu })$ with respect to $\mu$ when $\mu>0$ is a consequence of standard uniqueness/compactness arguments. The continuity at $\mu = 0$ is more involved, especially when $\alpha \geq \frac{1}{2}$. In this case, we have to prove that the contact point between the two domains $\Omega_1$ and $\Omega_2$ becomes a removable singularity. To achieve this result, we will use the extension on the upper half plane of the fractional Laplacian introduced by Caffarelli and Silvestre in \cite{Caffarelli:Silvestre:Extension}. For the case $\alpha < \frac{1}{2}$, the case $\mu = 0$ can be treated as the case $\mu >0$ thanks to the density of the set of functions satisfying $\phi = 0$ (see \cite{tartar}). We emphasise that the demonstration holds true up to $\mu = 0$ because we work in a one dimensional space. Indeed, in one dimension, there is only one way to connect two intervals. For $d>2$, the result holds true for distances $\mu>0$, with no real modification.
\subsubsection{ The limit $\alpha\to1$: Consequences of Theorem \ref{theorem_3} }

 It is well known (see \cite{Hitchhiker}) that for a smooth function $n$ and for all $x \in \mathbb{R}$, the function $\alpha\mapsto(-\Delta)^\alpha n(x)$ is continuous, and that
\[(-\Delta)^\alpha n(x) \underset{\alpha \rightarrow 1^-}{\longrightarrow} (-\Delta) n (x).\]
%Since the Laplacian, is a local operator, we expect that for fixed $\mu>0$, $\lambda_\alpha (\Omega_{1,2,\mu }) \geq 0$ for $\alpha$ closed enough to $1$. %Under our knowledge, none result about the monotonicity of the application $(\alpha \mapsto\lambda_{1,2,\mu, \alpha} +1)$ (the principal eigenvalue of the fractional Laplacian with Dirichlet condition outside $\Omega$) is known. 
The standard compactness/uniqueness argument yields  the continuity of the function $(\alpha \in ]0,1[ \rightarrow \xi_\alpha (\Omega_{1,2,\mu }))$. The next proposition describes, in our specific case, the dynamics of $ \xi_{\alpha}(\Omega_{1,2,\mu})$ when $\alpha$ tends to $1$.
\begin{proposition}\label{theorem_4}
The function $\alpha \in ]0,1[ \mapsto \xi_\alpha(\Omega_{1,2,\mu})$ is continuous up to $\alpha=1$ with 
\begin{equation}
\xi_1(\Omega_{1,2,\mu}) = \min \ (\underline{\xi}_1, \underline{\xi}_2).
\end{equation}
\end{proposition}
%\noindent The proof of Proposition \ref{theorem_4} uses the main result of \cite{non_local_to_local}. This article states that when we focus on a family of Poisson equation involving the $\alpha$-fractional Laplacian, under some assumptions we can pass to the limit $\alpha \rightarrow 1^-$. 
%We only provide the proof of the continuity of $\alpha = 1$ since for $\alpha < 1$, it relies on compact arguments similar to those used in the proof of Theorem \ref{theorem_3}. 

\noindent We now focus on the dynamics of $\xi_{\alpha}(\Omega_{1,2,\mu})$ when $( \mu, \alpha)$ converges both to $(0,1)$. There is a competition between the non-local character which becomes "less important" as $\alpha$ tends to $1$ and the requirement of this non-local character which becomes also "less important" as $\mu$ tends to $0$. Theorem \ref{theorem_3} and Proposition \ref{theorem_4} imply the following statement, which has once again no equivalent in the case of local diffusion:

\begin{theorem}\label{Propo_alpha_mu}
Under the previous hypothesis, for all $\xi_* \in [\xi_{1}(\Omega_{1,2,0}), \min (\underline{\xi}_{1,1}, \underline{\xi}_{2,1})]$, there exists a sequence $( \mu_k, \alpha_{k})_{k \in \mathbb{N}}$ such that 
\begin{equation}\label{lambdastar}
(\mu_k, \alpha_{k}) \underset{k \rightarrow +\infty}{\longrightarrow} (0,1) \ \text{ and } \ \xi_{\alpha_k}(\Omega_{1,2,\mu_k}) \underset{k \rightarrow +\infty}{\longrightarrow} \xi_*.
\end{equation}
\end{theorem}

%In the proof, we begin by constructing a sequence $(\mu_k, \alpha_k ) \rightarrow (0,1)$ such that $\xi_{\alpha_k}(\Omega_{1,2,\mu_k}) \rightarrow \xi_1(\Omega_{1,2,0})$. This sequence will be constructed with the help from Theorem \ref{theorem_3}. Next, we show thanks to Theorem \ref{theorem_4}, the existence of a sequence $(\nu_k, \beta_k) \rightarrow (0,1)$ such that $\xi_{\beta_k}(\Omega_{1,2,\nu_k}) \rightarrow \underline{\xi}_{i, 1}$. To finish, thanks to the continuity of $\xi_\alpha(\Omega_{1,2,\mu})$ with respect to $\mu$, we conclude with the intermediate value theorem. 

\subsection{Existence and uniqueness of a steady  solution to \eqref{stationary_state}}

%We focus on the existence and the uniqueness of bounded positive non-trivial solutions of  
%\begin{equation*}\tag{1}
%\left\lbrace
%\begin{aligned}
%&(-\Delta)^\alpha n_+ = n_+ - n_+^2 &&\text{ in } \Omega, \\
%&n_+ = 0 &&\text{ in } \Omega^c.
%\end{aligned}
%\right.
%\end{equation*}

%The aim of this work is to study the influence of the non-local aspect of the fractional Laplacian in patchy domains. As a consequence of these result on the principal eigenvalue in bounded patchy domains, we study the existence and the uniqueness of a non-trivial bounded solution of 

\subsubsection{Notations and assumptions on $\Omega$}
Before giving the hypothesis and the results, we introduce some notations which will be used all along the article. Let $\mathcal{O}$ be a general smooth domain of $\mathbb{R}^d$, $x$ be a point of $\mathbb{R}^d$ and $\nu$ be a positive constant. Then we define the sets:
\begin{equation}\label{sets}
\mathcal{O} + x = \left\lbrace y \in \mathbb{R}^d \text{ such that } y-x \in \mathcal{O}\right\rbrace \  \text{ and } \ \mathcal{O}_\nu = \left\lbrace y \in \mathcal{O} \text{ such that } \mathrm{dist}(y, \partial \mathcal{O}) > \nu \right\rbrace
\end{equation}
where $\partial \mathcal{O}$ is the boundary of $ \mathcal{O}$. Since the distance to the boundary of the  domain $\Omega$ under study will play an important role, we will denote it by $\delta$, 
\begin{equation}\label{{def_delta_chap4}}
\text{ i.e. }  \ \delta(x) = \mathrm{dist} (x, \partial \Omega ) 1_{\Omega}(x). 
\end{equation}
When it is defined, the principal eigenvalue of the fractional Dirichlet operator $(-\Delta)^\alpha - Id$ in $\mathcal{O}$ will play also an important role in the following. We will denote it by $\lambda_\alpha(\mathcal{O})$:
\begin{equation} \label{def_lambda}
\text{ i.e. }  \ 
\left\lbrace 
\begin{aligned}
(-\Delta)^\alpha \phi(x) - \phi(x) &= \lambda_\alpha(\mathcal{O}) \phi (x) && \text{ for } x \in \mathcal{O}, \\
\phi(x) &= 0 && \text{ for } x \in \mathbb{R}^d \backslash \mathcal{O}. 
\end{aligned}
\right.
\end{equation}
We underline that $\lambda_\alpha (\mathcal{O}) = \xi_\alpha (\mathcal{O}) - 1$. The principal eigenvalue $\lambda_\alpha(\mathcal{O})$ will be a key ingredient of most of the upcoming results. Note that a such eigenvalue is well defined for instance when the domain $\mathcal{O}$ is smooth and bounded with a finite number of components. It is also well defined if $\mathcal{O}$ is smooth, periodic and such that the number of components in all compact sets is finite. For such domains, there is a dichotomy: either $\lambda_\alpha(\mathcal{O})<0$ and \eqref{stationary_state} admits a positive bounded non-trivial solution, either $\lambda_\alpha (\mathcal{O}) \geq 0$ and the unique positive bounded solution of \eqref{stationary_state} is $0$.

\vspace{0.5cm}

We assume that the domain $\Omega$ may be written as
\begin{equation}\label{H1Chap4}\tag{H1}
\Omega = \underset{k \in \mathbb{N}}{\bigcup} \Omega_k
\end{equation}
where the sets $(\Omega_k)_{k \in \mathbb{N}}$ are smooth, connected, bounded. Moreover, we assume that 
\begin{equation}\label{H2Chap4} \tag{H2}
\Omega \text{ satisfies the uniform interior and exterior ball condition.}
\end{equation}
\begin{definition}[The interior and exterior ball condition]
A set $\mathcal{O}$ satisfies the uniform interior and exterior ball condition if there exists $\varepsilon_0 \in ]0, 1[$ such that 
\begin{equation}
\begin{aligned}
&\forall x \in \mathcal{O}, \ \exists z_x \in \mathcal{O} &&\text{ such that } x \in B(z_x, \varepsilon_0) \subset \mathcal{O}, \\
\text{ and  } \quad &\forall y \in \mathcal{O}^c, \ \exists z_y \in \mathcal{O}^c &&\text{ such that } y \in B(z_y, \varepsilon_0) \subset \mathcal{O}^c.
\end{aligned}
\end{equation}
\end{definition}
We assume that $\Omega$ can be decomposed in the following form : 
\begin{equation}\label{H3Chap4}\tag{H3}
\Omega = \Omega_- \cup   \Omega_+  = \underset{k \in \mathbb{N}}{\bigcup} ( \Omega_{-,k} \cup \Omega_{+,k}) .
\end{equation}
In the following, when we pick $x \in \Omega_\pm$, the integer $k_x$ will denote the only integer such that $x \in \Omega_{\pm,k_x}$. 

We assume that the domain $\Omega_+$ is composed by some uniformly bounded clusters $\mathcal{C}_{+,k}$ which are "far way" each other
\begin{equation}\label{H4Chap4}\tag{H4}
\begin{aligned}
\text{ i.e.} \qquad  \forall k \in \mathbb{N}, \ \exists z_k \in \mathbb{R}^d \text{ such that }& \quad \mathcal{C}_{+,k} \subset B(z_k, r_0), \quad  |z_i - z_j | > R_0+r_0 \ (\text{ for } i \neq j) \\  \text{ and }& \quad \mathrm{dist}(\Omega_+ , \Omega_-) > R_0 . 
\end{aligned}
\end{equation}
In what follows, the constant $R_0$ will be assume suitably large. Moreover, we assume that the eigenvalue of the Dirichlet operator $(-\Delta)^\alpha - Id$ in $\left\lbrace x \in \mathbb{R}^d | \mathrm{dist}(x, \mathcal{C}_{+,k} ) < \varepsilon_1 \right\rbrace$ is uniformly bounded from below for some positive $\varepsilon_1 < \frac{\varepsilon_0}{4}$:
\begin{equation}\label{H5Chap4}\tag{H5}
\text{ i.e. } \ 0 < \lambda_0 < \lambda_\alpha (\left\lbrace x \in \mathbb{R}^d | \mathrm{dist}(x, \mathcal{C}_{+,k} ) < \varepsilon_1 \right\rbrace). 
\end{equation}

For $\Omega_-$, we assume that it is not empty and there exists a finite number of bounded sets $(\underline{\Omega}_{-,n})_{n \in \left\lbrace 1,... , N\right\rbrace}$ made up with a finite number of connected components such that 
\begin{equation} \label{H6Chap4}\tag{H6}
\begin{aligned}
& \forall n \in \left\lbrace 1,..., N \right\rbrace , \   \lambda_\alpha (\underline{\Omega}_{ -,n})<0,\\
\text{ and } &\forall x \in \Omega_{-}, \exists( y, n) \in \mathbb{R}^d \times \left\lbrace 1,..., N \right\rbrace \text{ such that } x \in \left(y + \underline{\Omega}_{-,n} \right) \text{ and } (y + \underline{\Omega}_{-,n} ) \subset \Omega. 
\end{aligned}
\end{equation}

\begin{remark}
The assumptions \eqref{H1Chap4}--  \eqref{H6Chap4} on $\Omega$ cover a large case of sets from the sets with only one bounded connected component to a general unbounded with an infinite number of components. 
\end{remark}

\begin{figure}
    \centering
    \includegraphics[width=8cm,height=6cm]{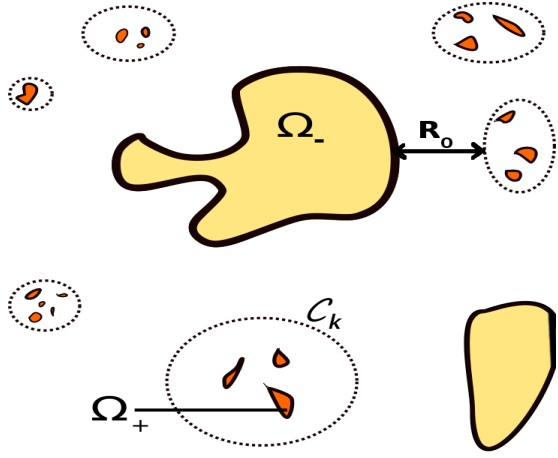}
    \caption{Illustration of a possible domain $\Omega$}
    \label{fig:my_label1chap4}
\end{figure}

At the end of section \ref{section2}, we present some examples of domains which satisfy such assumptions. 

\begin{remark}
The decomposition $\Omega = \Omega_- \cup \Omega_+$ is not unique. For instance if $\Omega_+$ is bounded, one can take $\Omega_-' = \Omega$ and $\Omega_+ = \emptyset$.
\end{remark}

\subsubsection{The steady solution $n_+$: existence, accurate estimates, uniqueness}

We begin by stating that there exists a bounded non-trivial solution $n_+$
\begin{theorem}
Under the assumption \eqref{H6Chap4}, there exists a bounded non-trivial solution of \eqref{stationary_state}.
\end{theorem}
\noindent Thanks to \eqref{H6Chap4}, we can provide a subsolution: the solution of \eqref{stationary_state} where the domain $\Omega$ is switched with $\Omega_- \cap B(0, R)$ (with $R$ large enough such that $\lambda_\alpha (\Omega_- \cap B(0,R))< 0$). Next, a supersolution is the constant function $1$. Finally, we construct a solution by iteration from this sub and super solution (see \cite{Smoller}).

\noindent A much less classical result is an estimate from above and below of any positive non-trivial bounded solutions of \eqref{stationary_state}.  
\begin{theorem}\label{lemma1}
Assume that $\Omega$ fulfils \eqref{H1Chap4}-- \eqref{H6Chap4} and that $R_0>C_{d,\alpha, \varepsilon_0, \varepsilon_1, r_0}$ where $C_{d,\alpha, \varepsilon_0, \varepsilon_1, r_0}$ is a positive constant depending on the parameters $d,\alpha, \varepsilon_0, \varepsilon_1, r_0$ then for any bounded positive solution $n_+$ of  \eqref{stationary_state}, there exists two positive constants $c_1 \in ]0,1[, \ C_1 \in \mathbb{R}_+$ such that for all $x\in \mathbb{R}^d$ we have 
\begin{equation}\label{lemmaeq1}
c_1 \mathcal{G}(x) \leq n_+(x) \leq \mathcal{G}(x)
\end{equation}
with 
\begin{equation}\label{G1}
\mathcal{G}(x) =  \min \left(\min(C_1 \delta(x)^\alpha, \varepsilon^\alpha) \times G(x), 1\right).
\end{equation}
and 
\begin{equation}
G(x) = 1_{\Omega_-}(x) +\left( \int_{\Omega_-} \frac{1}{|x-y|^{d+2\alpha}}dy  \right)  \  1_{\Omega_+}(x) .
\end{equation}
\end{theorem}
\noindent We underline that the function $\mathcal{G}$ depends only on the decomposition \eqref{H3Chap4}. We will devote a special section to its proof, its strategy being presented at the beginning. Thanks to Theorem \ref{lemma1}, we prove uniqueness of the solution of \eqref{stationary_state}. 
\begin{theorem}\label{theorem1Chap4}
Assuming \eqref{H1Chap4}-- \eqref{H6Chap4} and if $R_0$ is large enough then there exists a unique bounded positive non-trivial solution of \eqref{stationary_state}.
\end{theorem}
\noindent The proof of uniqueness follows a general argument introduced by Berestycki in \cite{Berestycki81}. Thanks to Theorem \ref{lemma1}, we compare two solutions and we conclude thanks to the maximum principle or the fractional Hopf Lemma.

\noindent A consequence is that the steady solution is a global attractor for the Cauchy Problem.
\begin{corollary}\label{propagation}
Under the previous hypothesis, if $n$ is solution of \eqref{equation_with_time} and $n_0$ is positive, non-trivial and compactly supported initial data with the closure of the support included in the closure of $\Omega$ then there holds
\[\forall x \in \mathbb{R}^d, \quad \underset{ t \rightarrow +\infty}{\lim} n(x,t) = n_+(x).\]
\end{corollary}
\noindent We underline that the proof will use the uniqueness result. From a biological point of view, we talk about a colonisation rather than an  invasion phenomena. Indeed, invasion imply a colonisation phenomenon and a an autonomy from the population in the neighborhood of the colonised area. This last fact does not hold in $\Omega_+$ according to Theorem \ref{lemma1}.

\subsection{Discussion, comparison with existing results}
Theorem \ref{theorem1Chap4} has a link with some recent results dealing with steady solutions of nonlocal Fisher-KPP equations in general environments. Closest to the result is the analysis of Berestycki, Coville and Vo 
\cite{BerCovVo}, where the dispersal is given by a smooth integral kernel, the domain is the whole space, but the reproduction term is inhomogeneous. Under essentially two assumptions, namely that a generalised principal eigenvalue (see  \cite{BerRossi} for the general definitions) is negative, and that the medium outside a ball is unfavourable, existence of uniqueness of a steady solution under a certain subsolution is established. We note a related work by Brasseur 
\cite{Bra} that achieves a passage to the limit of a more and more concentrated dispersal. Our general setting is less general than the situation considered in \cite{BerCovVo}, in particular we have to make assumptions related to the fact that the environment is fragmented. in particular, whether Theorem  \ref{theorem1Chap4}  holds under the sole assumption on the sign of a generalised principal eigenvalue is an important question whose answer is unknown to us at the moment.  Additional technical difficulties are present due to the fragmented environment (barriers at the boundary of the domain are sometimes tricky to devise) as well as the presence of the fractional Laplacian. On the other hand, we note that our uniqueness result holds in the whole class of bounded solutions due to the general estimate provided in Theorem \ref{lemma1}, which we regard as one of the main results of this work.

\noindent An aspect of the problem, that would call for further developments, is the detailed description of the invasion, in other words quantitative estimates on the convergence to the steady solution. This would be especially interesting if infinitely many patches are favourable - a description of the steady states in this situation is, by the way, still to be developped, although not totally out of reach from the arguments of the present work.   We mention the periodic setting treated in \cite{papier2}, where exponential invasion is proved. To understand how things have to be modified outside this setting is still, to our knowledge, open.
  \subsection{Outline of the paper}
\noindent We first provide in Section \ref{sign_lambda} a study on the one-dimensional case: the dependence of the sign of $\xi_\alpha (\Omega_{1,2,\mu})$ on the parameters $|\underline{\Omega}_{i,\alpha}|$, $\mu$ and $\alpha$. Next, in Section \ref{section2} the proof of Theorem \ref{lemma1}. Section \ref{sectionUniqueness} is devoted to the proof of Theorem \ref{theorem1Chap4}. In section \ref{section4}, we establish Theorem \ref{propagation}.  We also provide some numerics in order to illustrate Corollary \ref{theorem_3} and Theorem \ref{theorem_4}. \\
In order to be more readable, we do not write the principal eigenvalue $P.V.$ and the constant $C_\alpha$ in front of the fractional Laplacian. When there is no possible confusion, the constants denoted by $c$ and $C$ may change from one line to another. In all this work, when it is not precised, $\phi$ is usually used to denote a principal eigenfunction. Moreover, it is taken positive and with unit $L^2$ norm.

\section{Dependence of the sign of $\xi_\alpha (\Omega_{1,2,\mu})$ on the parameters}\label{sign_lambda}

In subsection \ref{dependence on Omega_i and mu}, we focus on the dependence of $\xi_\alpha(\Omega_{1,2,\mu})$ on the size of $\Omega_i$ and $\mu$. We provide the proof of Theorem \ref{theorem_3} and Proposition \ref{theorem_1}. Then, in subsection \ref{dependence_alpha}, we provide the proof of Theorem \ref{theorem_4}. Finally, subsection \ref{subsection2-3} is devoted to the investigation of the description of the dynamics of $\xi_{\alpha}(\Omega_{1,2,\mu})$ when $(\mu, \alpha) \rightarrow (0,1)$ at the same time. 

\subsection{Dependence on $|\Omega_i|$ and $\mu$}\label{dependence on Omega_i and mu}

We start the proof of Theorem \ref{theorem_3}  with the following classical  (and useful) bounds on $\xi_\alpha(\Omega_{1,2,\mu})$.
\begin{proposition}\label{theorem_1}
Under the previous hypothesis, there holds for all $\mu \geq 0$
\[\xi_\alpha (\Omega_{1,2,\mu}) < \min \ (\underline{\xi}_{1,\alpha}, \underline{\xi}_{2,\alpha} ).\]
\end{proposition}
\noindent \begin{proof}
For all $\mu>0$, it is straightforward that $H^\alpha_0(\Omega_i) \subset H^\alpha_0(\Omega)$. Thus, we deduce that the function $\underline{\phi}_{i,\alpha} \in H^\alpha_0(\Omega)$. We conclude thanks to the Rayleigh quotient that we have:
\[ \xi_\alpha (\Omega_{1,2,\mu }) = \underset{\phi \in H^\alpha_0(\Omega) \backslash \left\lbrace 0 \right\rbrace }{\min } R_\alpha (\phi) \leq R_\alpha (\underline{\phi}_{i,\alpha}) = \underline{\xi}_{i,\alpha} .\]
\end{proof}

\begin{proof}[Proof of Theorem \ref{theorem_3}]
We first show the monotonicity. Next, we prove the continuity for $\mu > 0$ finally we demonstrate the continuity up to $\mu = 0$. \\
\textbf{Proof of the monotonicity of $(\mu \mapsto  \xi_\alpha (\Omega_{1,2,\mu }))$.}  Let $\mu_1$ and $\mu_2$ be two positive constants such that $\mu_1 < \mu_2$. We will consider $\Omega_{1,2,\mu_1}$ and $\Omega_{1,2,\mu_2}$, that we write explicitely (possibly up to a translation) in order to fix  ideas:
\[\Omega_{1,2,\mu_1}=]-A_1-\mu_1, -\mu_1[\cup ] \mu_1, A_2+\mu_1[ \ \text{ and } \ \Omega_{1,2,\mu_2}=]-A_1-\mu_2,-\mu_2[\cup ] \mu_2, A_2+\mu_2[.\]
We define $\mu = \mu_2 - \mu_1$. We recall that for all $j \in \left\lbrace 1, 2 \right\rbrace$
\[\xi_\alpha (\Omega_{1,2,\mu_j }) = \underset{\phi \in H^\alpha_0(\Omega_{1,2,\mu_j}) \backslash \left\lbrace 0 \right\rbrace}{\min} R_\alpha (\phi)= R_\alpha (\phi_{1,2,\mu_j,\alpha}).\]
The idea is to translate each component of the support of $ \phi_{1,2,\mu_2,\alpha}$ in $\Omega_{1,2,\mu_1}$. Thus we define 

\begin{equation}\label{defini_mu_k}
\phi^\mu(x):= \left\{
    \begin{array}{ll}
			 \phi(x - \mu_2 + \mu_1) = \phi_{1,2,\mu_2,\alpha}(x-\mu) \										&\mbox{ if } x \in ]-\infty, 0[,\\
	          \phi(x + \mu_2 - \mu_1) = \phi(x+\mu) \ &\mbox{ if }x \in [0, +\infty[
     \end{array}
\right.
\end{equation} 
for any $\phi \in H^\alpha_0(\Omega_{1,2,\mu_2})$. We easily observe that $\phi^{\mu}$ belongs to $H^\alpha_0(\Omega_{1,2,\mu_1})$. Next, we remark that $\| \phi^\mu_{1,2,\mu_2,\alpha}\|_2 = \| \phi_{1,2,\mu_2,\alpha}\|_2 = 1$. The aim is to show that 
\begin{equation}\label{equation:lambda:mu:increasing}
[ \phi^\mu_{1,2,\mu_2,\alpha}]_\alpha \leq [ \phi_{1,2,\mu_2,\alpha} ]_{\alpha}.
\end{equation}
If \eqref{equation:lambda:mu:increasing} holds true, since the $L^2$ norm is conserved, we obtain from \eqref{def_R} that
\[ R_\alpha(\phi^\mu_{1,2,\mu_2,\alpha}) = \frac{[ \phi^\mu_{1,2,\mu_2,\alpha}]_\alpha }{\|\phi^\mu_{1,2,\mu_2,\alpha} \|_2^2}  \leq \frac{ [ \phi_{1,2,\mu_2,\alpha} ]_{\alpha}}{ \|\phi_{1,2,\mu_2,\alpha}\|_2^2} =  R_\alpha(\phi_{1,2,\mu_2,\alpha})  .\]
It allows us to conclude that 
\[\xi_\alpha (\Omega_{1,2,\mu_1 }) = \underset{\phi \in H^\alpha_0(\Omega_{1,2,\mu_1}) \backslash \left\lbrace 0 \right\rbrace}{\min} R_\alpha (\phi) \leq R_\alpha ( \phi^\mu_{1,2,\mu_2,\alpha} ) \leq R (\phi_{1,2,\mu_2,\alpha}) = \xi_\alpha (\Omega_{1,2,\mu_2 }).\]

Thus, we prove \eqref{equation:lambda:mu:increasing}. We will denote 
\[k_\varepsilon (x) = \max (\varepsilon, |x|^{1+2\alpha}) \text{ such that } [ \phi]_\alpha =\underset{\varepsilon \rightarrow 0}{\lim} { \ \frac{1}{2} \int_\mathbb{R} \int_\mathbb{R} \frac{(\phi(x) - \phi(y))^{2}}{k_\varepsilon(x-y)}dydx}.\]
For all $\varepsilon >0$ and for all positive $\phi \in H^\alpha_0(\Omega_{1,2,\mu_2})$ thanks to the Fubini-Tonelli theorem, we have 
\begin{align*}
\frac{1}{2}\int_\mathbb{R} \int_\mathbb{R} \frac{(\phi^\mu(x) - \phi^\mu(y))^{2}}{k_\varepsilon(x-y)}dydx & = \frac{1}{2} \int_\mathbb{R} \int_\mathbb{R} \frac{\phi^\mu(x)^2 + \phi^\mu(y)^{2} - 2 \phi^\mu(x)\phi^\mu(y)}{k_\varepsilon(x-y)}dydx \\
&=   \int_\mathbb{R} \phi^\mu(x)^2  \int_\mathbb{R} \frac{1}{k_\varepsilon(x-y)}dydx -   \int_\mathbb{R} \int_\mathbb{R} \frac{ \phi^\mu(x)\phi^\mu(y)}{k_\varepsilon(x-y)}dydx \\
&=  C_\varepsilon \int_\mathbb{R} \phi^\mu(x)^2 dx -   \int_{\Omega_{1,2,\mu_1}} \int_{\Omega_{1,2,\mu_1}} \frac{ \phi^\mu(x)\phi^\mu(y)}{k_\varepsilon(x-y)}dydx .
\end{align*}
Here, $C_\varepsilon = \int_{\mathbb{R}} \frac{1}{k_\varepsilon(z)}dz$ is a constant independent of the choice of $\phi$. Next we have :
\begin{equation}\label{pf_lambda_mu_croissant_eq1}
\begin{aligned}
\int_\mathbb{R} \int_\mathbb{R} \frac{(\phi^\mu(x) - \phi^\mu(y))^{2}}{k_\varepsilon(x-y)}dydx & =  C_\varepsilon \int_\mathbb{R} \phi^\mu(x)^2 dx  -\int_{\Omega^{\mu_1}_1} \int_{\Omega^{\mu_1}_1} \frac{ \phi^\mu(x)\phi^\mu(y)}{k_\varepsilon(x-y)}dydx \\
&- \int_{\Omega^{\mu_1}_2} \int_{\Omega^{\mu_1}_2} \frac{ \phi^\mu(x)\phi^\mu(y)}{k_\varepsilon(x-y)}dydx  - 2  \int_{\Omega^{\mu_1}_2} \int_{\Omega^{\mu_1}_1} \frac{ \phi^\mu(x)\phi^\mu(y)}{k_\varepsilon(x-y)}dydx  
\end{aligned}
\end{equation}
Thanks to a change of variable, we obtain that for all $i \in \left\lbrace 1,2 \right\rbrace $:
\begin{equation}\label{pf_lambda_mu_croissant_eq2}
\int_{\Omega^{\mu_1}_i} \int_{\Omega^{\mu_1}_i} \frac{ \phi^\mu(x)\phi^\mu(y)}{k_\varepsilon(x-y)}dydx =  \int_{\Omega^{\mu_2}_i} \int_{\Omega^{\mu_2}_i} \frac{ \phi(x)\phi(y)}{k_\varepsilon(x-y)}dydx.
\end{equation}
Since for $(x, y) \in \Omega^{\mu_2}_1 \times \Omega^{\mu_2}_2$ we have $ k_\varepsilon (x-y+2\mu) \leq k_\varepsilon(x-y) $, we obtain 
\begin{equation}\label{pf_lambda_mu_croissant_eq3}
\int_{\Omega^{\mu_2}_2} \int_{\Omega^{\mu_2}_1} \frac{ \phi(x)\phi(y)}{k_\varepsilon(x-y)}dydx \leq \int_{\Omega^{\mu_2}_2} \int_{\Omega^{\mu_2}_1} \frac{ \phi(x)\phi(y)}{k_\varepsilon(x-y+2\mu)}dydx= \int_{\Omega^{\mu_1}_2} \int_{\Omega^{\mu_1}_1} \frac{ \phi^\mu(x)\phi^\mu(y)}{k_\varepsilon(x-y)}dydx.
\end{equation}
Inserting \eqref{pf_lambda_mu_croissant_eq2} and \eqref{pf_lambda_mu_croissant_eq3} in \eqref{pf_lambda_mu_croissant_eq1}, we conclude that for all $\varepsilon>0$ and all positive $\phi \in H^\alpha_0(\Omega_{1,2,\mu_2})$, we have 
\[ \int_\mathbb{R} \int_\mathbb{R} \frac{(\phi^\mu(x) - \phi^\mu(y))^{2}}{k_\varepsilon(x-y)}dydx \leq \int_\mathbb{R} \int_\mathbb{R} \frac{(\phi(x) - \phi(y))^{2}}{k_\varepsilon(x-y)}dydx. \]
Sending  $\varepsilon$ to 0, we conclude that \eqref{equation:lambda:mu:increasing} holds true and the conclusion follows.

\vspace{0.5cm}

\noindent \textbf{Proof of the continuity of $(\mu \mapsto \xi_\alpha (\Omega_{1,2,\mu }))$ for $\mu>0$.} Let $\mu>0$ and $(\mu_k)_{k \in \mathbb{N}} $ with $\mu_k \underset{k \rightarrow +\infty}{\longrightarrow} \mu$ and $\mu_k >0$. According to Proposition \ref{theorem_1} and the Krein-Rutman Theorem, we have $0 \leq \xi_\alpha(\Omega_{1,2,\mu}) < \min (\underline{\xi}_1, \underline{\xi}_2)$. Up to a subsequence, $\xi_\alpha (\Omega_{1,2,\mu_k })$ converges to $\xi_\infty$. We normalise $\phi_{1,2,\mu_k,\alpha}$ such that $\|\phi_{1,2,\mu_k,\alpha}\|_{L^2}=1$. Next, thanks to the Rayleigh quotient, we obtain that 
\[ [\phi_{1,2,\mu_k,\alpha} ]_\alpha \leq \xi_\alpha (\Omega_{1,2,\mu_k }) \leq \min (\underline{\xi}_1, \underline{\xi}_2). \]
Thus, we deduce that $\| \phi_{1,2,\mu_k,\alpha} \|_{H^\alpha} $ is bounded. Up to a new subsequence, $\phi_{1,2,\mu_k,\alpha}$ converges strongly in $L^2(\mathbb{R})$ and weakly in $H^\alpha(\mathbb{R})$ to $\phi_\infty$. It is straightforward to obtain that $\phi_\infty\geq 0$ and $\phi_\infty =0 $ in $(\Omega_{1,2,\mu})^c$. Moreover, since $\Omega_{1,2,\mu_k}$ converges to $\Omega_{1,2,\mu}$ we deduce that for all compact set $K$ of $\Omega_{1,2,\mu}$, there exists $k_0 \in \mathbb{N}$ such that for $k>k_0$, we have 
\[K \subset \Omega_{1,2,\mu_k} \text{ and } 1-\varepsilon \leq  \|\phi_{1,2,\mu_k,\alpha} \|_{L^2(K)} \]
with $\varepsilon$ as small as we want. We deduce that $\| \phi_\infty \|_{L^2(\Omega_{1,2,\mu})} =1$.  With the same idea, we get that in all compact set $K$ of $\Omega_{1,2,\mu}$, $\phi_\infty$ is a weak solution to 
\begin{equation*}
\left\lbrace 
\begin{aligned}
(-\Delta)^\alpha \phi_\infty & = \xi_\infty \phi_\infty &&  \text{ for } x \in K,  \\
\phi_\infty &= 0 && \text{ for } x \in \Omega_{1,2,\mu}^c, \\
\phi_\infty \geq 0,& \ \|\phi_\infty\|_{L^2} = 1.
\end{aligned}
\right.
\end{equation*}
Since it is true in all compact set of $\Omega_{1,2,\mu}$ we conclude that
\begin{equation}\label{phiinfty}
\left\lbrace 
\begin{aligned}
(-\Delta)^\alpha \phi_\infty & = \xi_\infty \phi_\infty &&  \text{ for } x \in \Omega_{1,2,\mu}, \\
\phi_\infty &= 0 && \text{ for } x \in \Omega_{1,2,\mu}^c, \\
\phi_\infty \geq 0,& \ \|\phi_\infty\|_{L^2} = 1.
\end{aligned}
\right.
\end{equation}
Thanks to the fractional elliptic regularity (see \cite{Caffarelli:Silvestre:Regularity}), we obtain that \eqref{phiinfty} is true in the strong sense. By uniqueness of the eigenvalue associated to a strictly positive eigenfunction, we finally conclude that $\xi_\infty = \xi_\alpha (\Omega_{1,2,\mu })$ and $\phi_\infty = \phi_{1,2,\mu, \alpha}$. 

\bigbreak

\noindent \textbf{Proof of the continuity up to $\mu=0$.} Let $(\mu_k)_{k \in \mathbb{N}}$ be a sequence such that $\mu_k \underset{ k \rightarrow +\infty}{\longrightarrow} 0$ and $\mu_k > 0$. Following the same idea than in the previous part, we find that $\xi_\alpha (\Omega_{1,2,\mu_k })$ converges to some $\xi_\infty$ and $\phi_{1,2,\mu_k,\alpha}$ converges to $\phi_\infty$ with $\phi_\infty$ a bounded solution of 
 \begin{equation}\label{phiinftydeux}
\left\lbrace 
\begin{aligned}
(-\Delta)^\alpha \phi_\infty & = \xi_\infty \phi_\infty &&  \text{ for } x \in ]-A,A[ \backslash \left\lbrace 0 \right\rbrace, \\
\phi_\infty &= 0 && \text{ for } x \in ]-A,A[^c, \\
\phi_\infty \geq 0,& \ \|\phi_\infty\|_{L^2} = 1.
\end{aligned}
\right.
\end{equation}
We consider two cases: $\alpha < \frac{1}{2}$ and $\alpha \geq \frac{1}{2}$. \\
\textbf{Case 1 $\alpha < \frac{1}{2}.$ } It is sufficient to remark that $\phi_{1,2,\mu_k,\alpha}$ converges to $\phi_\infty$ in $H^\alpha( ]-A,A[)$ (see the comments after the proof of Lemma 16.1 p. 82 of \cite{tartar}). Indeed, since the set of functions $\left\lbrace \psi \in H^\alpha_0( ]-A,A[) | \psi(0) = 0 \right\rbrace$ is dense in $H^\alpha_0( ]-A,A[)$, we conclude by compactness as in the case $\mu>0$. \\% (i.e. when $\mu > 0$). \\
\textbf{Case 2 $\alpha \geq \frac{1}{2}$.  } The idea is to show that 0 is a removable singularity in the extended problem, as as introduced in \cite{Caffarelli:Silvestre:Extension}. This will allow us to conclude that $\phi_\infty = \phi_{1,2,0,\alpha}$ and $\xi_\infty = \xi_\alpha (\Omega_{1,2,0 })$. The inspiration for the whole proof comes from Serrin \cite{Serrin}. \\
So, let $v$ be the solution of 
\begin{equation*}
\left\lbrace 
\begin{aligned}
-\mathrm{div} (y^a \nabla v ) &= 0 && \text{ for } (x,y) \in \mathbb{R}\times \mathbb{R}^+ \\  % \backslash \left\lbrace (0,0) \right\rbrace 
v(x, 0) &= \phi_\infty(x) && \text{ for } x \in \mathbb{R}, 
\end{aligned}
\right.
\qquad \text{ with } a= 1 -2\alpha.
\end{equation*}
%with $a = 1-2\alpha$. 
From \cite{Caffarelli:Silvestre:Extension} we have that, for all $x \in ]-A,A[ \backslash \left\lbrace 0 \right\rbrace$:
\[\underset{ y \rightarrow 0}{\lim}\  y^a\partial_y  v(x,y) =  - (-\Delta)^\alpha \phi_\infty (x) = - \xi_\infty \phi_\infty(x) = - \xi_\infty v(x,0).\]
We define $w$ as the solution of the following equation 
\begin{equation}\label{equation_petitw}
\left\lbrace 
\begin{aligned}
-\mathrm{div} (y^a \nabla w ) &= 0 &&\text{ for } (x,y) \in ]-L,L[ \times ]0,L[ \\
w(x, y) &= v(x,y)  &&\text{ for } (x,y) \in \partial ( ]-L,L[ \times ]0,L[ ) \cap \left\lbrace y >0 \right\rbrace,\\
\underset{ y \rightarrow 0}{\lim} \ y^a \partial_y w(x,y)&=  -\underset{ y  \rightarrow 0 }{\lim} \  \xi_\infty \  w (x,y) && \text{ for } x \in ]-L, L[,
\end{aligned}
\right.
\end{equation}
where $L \in ]0, A [$. That $w$  exists and is unique is a consequence of the implicit function Theorem   for $L$ small enough. Define $W=v-w$ such that 
\begin{equation}\label{equation_W}
\left\lbrace 
\begin{aligned}
-\mathrm{div} (y^a \nabla W ) &= 0 &&\text{ for } (x,y) \in ]-L,L[ \times ]0,L[\\
W(x, y) &= 0  &&\text{ for } (x,y) \in \partial(]-L,L[ \times ]0,L[)\cap \left\lbrace y >0\right\rbrace,\\
\underset{ y \rightarrow 0}{\lim} \ y^a \partial_y W(x,y)& =  -\underset{ y  \rightarrow 0 }{\lim} \ \xi_\infty \  W (x,y) && \text{ for } x \in ]-L,L[ \backslash \left\lbrace 0 \right\rbrace. 
\end{aligned}
\right.
\end{equation}
We are going to prove that $W=0$. % More precisely, we show that for all $\mu> 0$, 
%\[ \sup | W | < \mu. \]
For this purpose, we split $W$ into two parts: $W = W_1 + W_2$. The function $W_1$ is solution of the following equation:
\begin{equation}\label{equation_W1}
\left\lbrace 
\begin{aligned}
-\mathrm{div} (y^a \nabla W_1 ) &= 0 &&\text{ for } (x,y) \in B(0, \varepsilon)^+, \\
W_1(x, y) &= W(x,y)  &&\text{ for } (x,y) \in \partial B(0,\varepsilon)^+,\\
\underset{ y \rightarrow 0}{\lim} \ y^a \partial_y W_1(x,y)& =  0  && \text{ for } x \in  ]-\varepsilon,\varepsilon[
\end{aligned}
\right.
\end{equation}
where $\varepsilon \in ]0, \frac{L}{2}[$ will be chosen later on.\\% First note that we can extend \eqref{equation_W1} to the whole domain $B(0, \varepsilon)$ (see Lemme 4.1 in \cite{Caffarelli:Silvestre:Extension}). \\% Moreover, since $W$ is bounded, the point $(0,0)$ is a removable singularity of $W_1$ when $\varepsilon = 0 $ (see Theorem 7.36 in \cite{Potential_theory}). Following the general idea of Serrin in \cite{Serrin}, we deduce that
%\begin{equation}\label{W1}
%W_1 \underset{ \varepsilon \rightarrow 0 }{\longrightarrow } 0.
%\end{equation}
Next, we focus the study of $W_2$ on the domain $\left( ]-L,L[ \times ]0,L[ \right) \backslash B(0,\varepsilon)^+$. The equation for $W_2$ is %Since $W_2 	= W - W_1$, we have
\begin{equation}\label{equation_W2}
\left\lbrace 
\begin{aligned}
-\mathrm{div} (y^a \nabla W_2 ) &= 0 &&\text{ for } (x,y) \in ( ]-L,L[ \times ]0,L[ )\backslash B(0, \varepsilon)^+ \\
W_2(x, y) &= 0  &&\text{ for } (x,y) \in \partial (]-L,L[ \times ]0,L[ ) \cap \left\lbrace y > 0 \right\rbrace,\\
W_2 (x,y) &= W_1(x,y) && \text{ for } (x,y) \in \partial B(0, \varepsilon)^+\\
\underset{ y \rightarrow 0}{\lim} \ y^a \partial_y W_2(x,y)& =  -\underset{ y  \rightarrow 0 }{\lim} \  \xi_\infty W (x,y)  && \text{ for } x \in ]-L,L[ \backslash ]-\varepsilon, \varepsilon[. 
\end{aligned}
\right.
\end{equation}
We underline that $W_2 \underset{ \varepsilon \rightarrow  0 }{\longrightarrow} W$ weakly. Therefore, since $W_1  = W - W_2$, we deduce that $W_1 \underset{ \varepsilon \rightarrow 0}{\longrightarrow } 0$ weakly also.
%\bigbreak
%\textbf{Cas particulier $\alpha = \frac{1}{2}$: $a = 0$. }
%
%We find 
%\begin{equation}\label{equation_W}
%\left\lbrace 
%\begin{aligned}
%-\mathrm{div} (y^a \nabla W_2 ) &= 0 &&\text{ for } (x,y) \in B(0, A) \backslash B(0, \varepsilon) \\
%W_2(x, y) &= 0  &&\text{ for } (x,y) \in \partial B(0,A) \cap \left\lbrace y >0\right\rbrace,\\
%W_2 (x,y) &= 0 && \text{ for } (x,y) \in \partial B(0, \varepsilon)^+\\
%\underset{ y \rightarrow 0}{\lim} \  -\partial_y W_2(x,y)& =  \underset{ y  \rightarrow 0 }{\lim} (\lambda_\infty+1) W (x,y) + \underset{ y \rightarrow 0}{\lim} \  \partial_y W_1(x,y) && \text{ for } x \in ]-A,A[ \backslash \left\lbrace 0 \right\rbrace. 
%\end{aligned}
%\right.
%\end{equation}
In the following, we denote by 
\begin{align*}
    &\mathcal{D} = ]-L,L[ \times ]0,L[ \backslash B(0, \varepsilon) , &&\qquad  \Gamma_1 = \partial \left(]-L,L[ \times ]0,L[ \right) \  \backslash \ \left( ]-L,L[ \times \left\lbrace 0 \right\rbrace \right), \\
    &\Gamma_2 = ]-L, -\varepsilon[ \cup ]\varepsilon,L[, \hspace{2cm} \text{ and } &&\qquad  \Gamma_3 = \partial B(0, \varepsilon)^+.
\end{align*}

\begin{figure}[h!]
    \centering
    \includegraphics[width=5cm,height=5cm]{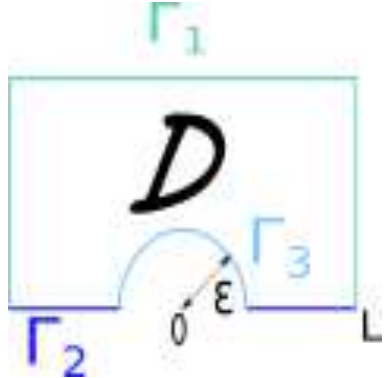}%[width=12cm, hight=5cm]
    \caption{The set under study}
    \label{fig:my_labelchap4}
\end{figure}

The aim of the rest of the proof is to prove that $W_2 \underset{\varepsilon \rightarrow 0}{\longrightarrow } 0 $. If it holds true, the conclusion follows. For this purpose, we use the superposition principle and we split the study of $W_2$ into two parts : $W_2'$ and $W_2''$ which are the respective solutions of

\begin{equation}\label{equation_W2P}
\left\lbrace 
\begin{aligned}
-\mathrm{div} (y^a \nabla W_2' ) &= 0 &&\text{ for } (x,y) \in \mathcal{D} \\
W_2'(x, y) &= 0  &&\text{ for } (x,y) \in \Gamma_1,\\
W_2' (x,y) &= W_1(x,y) && \text{ for } (x,y) \in \Gamma_3\\
\underset{ y \rightarrow 0}{\lim} \ y^a \partial_y W_2'(x,y)& = 0&& \text{ for } x \in \Gamma_2
\end{aligned}
\right.
\end{equation}
and 
\begin{equation}\label{equation_W2PP}
\left\lbrace 
\begin{aligned}
-\mathrm{div} (y^a \nabla W_2'' ) &= 0 &&\text{ for } (x,y) \in \mathcal{D} \\
W_2''(x, y) &= 0  &&\text{ for } (x,y) \in \Gamma_1,\\
W_2'' (x,y) &= 0 && \text{ for } (x,y) \in \Gamma_3,\\
\underset{ y \rightarrow 0}{\lim} \ y^a \partial_y W_2''(x,y)& =  -\underset{ y  \rightarrow 0 }{\lim}  \ \xi_\infty W (x,y)  && \text{ for } x \in \Gamma_2. 
\end{aligned}
\right.
\end{equation}
Since, $W_1$ tends weakly to $0$ as $\varepsilon \rightarrow 0$, it follows that $W_2'$ tends also weakly to $0$. It remains to prove that $W_2''$ vanishes.\\
Let $\phi$ be a test function in $H^1_0(\mathcal{D}, |y|^a )$ % \textcolor{red}{Je ne suis pas sûr de l'espace fonctionnel, L'espace des fonctions ne doit pas s'annuler sur $\left\lbrace y = 0 \right\rbrace$ mais c'est là où la mesure est nulle.}
%\[i.e. \phi \in \mathcal{H}\quad  \text{ with }\mathcal{H}:=  \left\lbrace \phi \in  | \ \mathrm{Tr}(\phi |_{ \Gamma_1})  = 0 \right\rbrace\]
(see Chapter 1 of \cite{Potential_theory} for the general framework of weighted Sobolev spaces). Noticing that for all $\tau >0$ and $(x,y) \in \mathcal{D} \cap \left\lbrace y > \tau \right\rbrace$, we have 
\begin{align*}
\mathrm{div}(\phi |y|^a\nabla W_2'')(x,y) &= \nabla \phi(x,y)  |y|^a \nabla W_2''(x,y)+ \phi(x,y) \mathrm{div}(|y|^a \nabla W_2'')(x,y)\\
 &=  \nabla \phi(x,y)  |y|^a \nabla W_2''(x,y). 
\end{align*}
We deduce that 
\begin{align*}
\int_\mathcal{D} \nabla \phi |y|^a \nabla W_2'' dxdy &= \int_{\mathcal{D} \cap \left\lbrace y > \tau  \right\rbrace} \nabla \phi |y|^a \nabla W_2'' dxdy + \int_{\mathcal{D} \cap \left\lbrace y \leq  \tau  \right\rbrace} \nabla \phi |y|^a \nabla W_2'' dxdy \\
&= \int_{\mathcal{D} \cap \left\lbrace y > \tau  \right\rbrace}\mathrm{div}(\phi |y|^a\nabla W_2'')(x,y) dxdy + \int_{\mathcal{D} \cap \left\lbrace y \leq  \tau  \right\rbrace} \nabla \phi |y|^a \nabla W_2'' dxdy \\
&= \int_{\partial(\mathcal{D} \cap \left\lbrace y > \tau  \right\rbrace)}\phi |y|^a\nabla W_2''(x,y) \cdot\nu  dx + \int_{\mathcal{D} \cap \left\lbrace y \leq  \tau  \right\rbrace} \nabla \phi |y|^a \nabla W_2'' dxdy \\
&= \int_{\mathcal{D} \cap \left\lbrace y = \tau  \right\rbrace}-\phi(x,\tau) |\tau|^a\partial_y W_2''(x,\tau)  dx + \int_{\mathcal{D} \cap \left\lbrace y \leq  \tau  \right\rbrace} \nabla \phi |y|^a \nabla W_2'' dxdy .
\end{align*}
Next, if we let $\tau$ tends to $0$,   the second term tends to 0 thanks to the Cauchy-Schwarz inequality. Recalling that $W = W_1+ W_2' + W_2''$, we finally have obtained the following variational equation:
\begin{equation}\label{var_form}
  \int_\mathcal{D} \nabla \phi\nabla W_2''  |y|^a dxdy - \int_{\Gamma_2}\phi(x,0) \xi_\infty W_2''(x,0) dx = \int_{\Gamma_2} \xi_\infty (W_1(x,0) + W_2'(x,0))\phi(x,0) dx .
\end{equation}
In order to prove the existence and uniqueness of a solution of \eqref{var_form}, we are going to apply the Lax-Milgram theorem. The linear map is the following : 
\[ \Lambda: \phi \in  H^1_0(\mathcal{D} , |y|^a) \rightarrow \xi_\infty \int_{\Gamma_2}  (W_1(x,0) + W_2'(x,0))\phi(x,0) dx .\]
%This map is continuous thanks to the continuity of the trace in $\mathcal{H}$.% (see Proposition 2.1 in \cite{Trace_cont}). 
Since $(W_1 + W_2') $ tends weakly to $0$ as $\varepsilon \rightarrow 0$, we deduce that 
\begin{equation}\label{in_L}
\| \Lambda \|_{H^1_0(\mathcal{D})'} \underset{\varepsilon \rightarrow 0}{\longrightarrow} 0. 
\end{equation}
Next, we show that for $L$ small enough, the bilinear form 
\[b:(u,v) \in H^1_0(\mathcal{D}) \times H^1_0(\mathcal{D}) \rightarrow \int_\mathcal{D} \nabla u \nabla v  |y|^a dxdy - \int_{\Gamma_2}u (x,0) \xi_\infty v(x,0) dx\]
is continuous and coercive.\\
First, the Poincar\'{e} inequality implies that $\phi\mapsto (\int_\mathcal{D} \|\nabla \phi \|^2|y|^adxdy)^{\frac{1}{2}}$ is a norm that is equivalent to the usual norm on $H^1_0(\mathcal{D},|y|^a) $ (see equation (1.5) p. 9 of \cite{Potential_theory}). thus the continuity follows.\\
Secondly, we prove that for $L$ small enough, $b$ is coercive. Indeed, since $a = 1-2\alpha < 0$, we have 
\begin{align*}
 \int_{\Gamma_2} u(x,0)^2 dx &= 2\int_{\Gamma_2} \int_{0}^L u(x,y) \partial_{y}u(x,y) \frac{|y|^a}{|y|^a} dy dx \\
 & \leq \frac{2}{L^a} \int_{\Gamma_2} \left( \int_0^L u(x,y)^2 |y|^a dy \right)^{\frac{1}{2}}  \left(\int_0^L \partial_y u(x,y)^2 |y|^a dy \right)^{\frac{1}{2}} dx \\
 &\leq \frac{2}{L^a} \int_{\Gamma_2} \left( \int_0^L u(x,y)^2 |y|^a dy \right)^{\frac{1}{2}}  \left(\int_0^L \partial_y u(x,y)^2 |y|^a dy \right) dx \\
 & \leq \frac{2}{L^a} \left(\int_{\Gamma_2} \int_0^L u(x,y)^2 |y|^a dy dx \right)^{\frac{1}{2}} \left( \int_{\Gamma_2} \int_0^L \partial_y u(x,y)^2 |y|^a dy dx \right) ^\frac{1}{2} \\
 &\leq 2C \mathrm{diam}(\mathcal{D})L^{2\alpha - 1} \int_\mathcal{D} |\nabla  u|^2 |y|^a dxdy .
 \end{align*}
%\textcolor{red}{Ce dernier calcul n'est pas vrai pour $\alpha < \frac{1}{2}$.}
Therefore, for $L$ small enough, we deduce that there exists $c_b>0$ such that  
\[b(u,u)>c_b \int_\mathcal{D} |\nabla u |^2 |y|^adxdy. \]
We conclude thanks to the Lax-Milgram theorem that there exists a unique solution to \eqref{var_form}. Moreover thanks to the estimates of the norm of the solution in the Lax-Milgram theorem and \eqref{in_L}, we deduce that 
\[\| W_2'' \|_{H^1_0(\mathcal{D})} \leq \frac{\|\Lambda\|_{{H^1_0(\mathcal{D})}'}}{c_b} \underset{ \varepsilon \rightarrow 0 }{\longrightarrow } 0.\]
The conclusion follows.

\end{proof}

% \begin{lemma}\label{Lemma2}
% There exists a distance $\mu_1$ such that $\lambda_\alpha (\Omega_{1,2,\mu_1 })>0$. 
% \end{lemma}

% Since the proof of Lemma \ref{Lemma2} is computational, we postpone it to Appendix B We finish by giving the proof of Corollary \ref{theorem_2}. 

% \begin{proof}[Proof of Corollary \ref{theorem_2}]
% The proof of \textit{(i)} is a consequence of the monotonicity and the continuity of $(\mu \mapsto \lambda_\alpha (\Omega_{1,2,\mu }))$ up to $0$ and assumption \eqref{H5Chap4}.\\
% The proof of \textit{(ii)} is a consequence of the monotonicity, the continuity of $(\mu \mapsto \lambda_\alpha (\Omega_{1,2,\mu }))$ and Lemma \ref{Lemma2}.
% \end{proof}

\subsection{Dependence on $\alpha$}\label{dependence_alpha}

In this subsection, we provide the proof of Proposition \ref{theorem_4}. 

\begin{proof}[Proof of Proposition \ref{theorem_4}] The continuity for $\alpha < 1$ is based on similar arguments as those presented in the proof of the continuity of $\xi_\alpha (\Omega_{1,2,\mu})$ for $\mu > 0$. Therefor, we focus on the continuity up to $\alpha = 1$.\\
First, we establish a monotonicity result. We claim that the function 
\[(\alpha \in ]0,1] \mapsto \mathrm{diam}(\Omega_{1,2,\mu})^{2\alpha} \xi_\alpha(\Omega_{1,2,\mu})) \text{ is increasing,}\]
(where $\mathrm{diam}(\Omega_{1,2,\mu})$ designates the diameter of $\Omega_{1,2,\mu}$). Indeed, since if $\alpha < \beta$ there holds $H^\beta_0(\Omega_{1,2,\mu}) \subset H^\alpha_0(\Omega_{1,2,\mu})$, we deduce that
\begin{align*}
     \xi_\alpha (\Omega_{1,2,\mu}) &\leq \underset{ \phi \in H^\alpha_0(\Omega_{1,2,\mu})}{\min } \mathcal{R}_\alpha(\phi) \\
     &\leq  \mathcal{R}_\alpha (\phi_{1,2,\mu, \beta}) \\
     &= \frac{1}{2}\int_\mathbb{R} \int_\mathbb{R} \frac{(\phi_{1,2,\mu, \beta}(x) - \phi_{1,2,\mu, \beta}(y))^2 }{|x-y|^{d+2\alpha}}dydx \\
     &= \frac{1}{2}\int_\mathbb{R} \int_{\mathbb{R} , |x-y|< \mathrm{diam}(\Omega_{1,2,\mu})} \frac{(\phi_{1,2,\mu, \beta}(x) - \phi_{1,2,\mu, \beta}(y))^2 }{|x-y|^{d+2\alpha}}dydx  \\
     &+ \frac{1}{2}\int_\mathbb{R} \int_{\mathbb{R} , |x-y|\geq \mathrm{diam}(\Omega_{1,2,\mu})} \frac{(\phi_{1,2,\mu, \beta}(x) - \phi_{1,2,\mu, \beta}(y))^2 }{|x-y|^{d+2\alpha}}dydx \\
    &= \frac{\mathrm{diam}(\Omega_{1,2,\mu})^{2(\beta - \alpha)}}{2}\int_\mathbb{R} \int_{\mathbb{R} , |x-y|< \mathrm{diam}(\Omega_{1,2,\mu})} \frac{(\phi_{1,2,\mu, \beta}(x) - \phi_{1,2,\mu, \beta}(y))^2 }{|x-y|^{d+2\beta}}dydx  \\
     &+ \frac{1}{2}\int_\mathbb{R} \phi_{1,2,\mu, \beta}(x) \int_{\mathbb{R} , |x-y|\geq \mathrm{diam}(\Omega_{1,2,\mu})} \frac{1 }{|x-y|^{d+2\alpha}}dydx \\
     &= \frac{\mathrm{diam}(\Omega_{1,2,\mu})^{2(\beta - \alpha)}}{2}\int_\mathbb{R} \int_{\mathbb{R} , |x-y|< \mathrm{diam}(\Omega_{1,2,\mu})} \frac{(\phi_{1,2,\mu, \beta}(x) - \phi_{1,2,\mu, \beta}(y))^2 }{|x-y|^{d+2\beta}}dydx  \\
     &+ \frac{\mathrm{diam}(\Omega_{1,2,\mu})^{-2\alpha}}{2}\int_\mathbb{R} \phi_{1,2,\mu, \beta}(x) ^2dx \\
     &= \frac{\mathrm{diam}(\Omega_{1,2,\mu})^{2(\beta - \alpha)}}{2}\int_\mathbb{R} \int_{\mathbb{R} } \frac{(\phi_{1,2,\mu, \beta}(x) - \phi_{1,2,\mu, \beta}(y))^2 }{|x-y|^{d+2\beta}}dydx  \\
     &= \mathrm{diam}(\Omega_{1,2,\mu})^{2(\beta - \alpha)} \xi_\beta(\Omega_{1,2,\mu}).
%     \underset{ \phi \in H^\beta_0(\Omega_{1,2,\mu})}{\min } \mathcal{R}(\phi) = \xi_\beta(\Omega_{1,2,\mu}).
\end{align*}
(Note that this last result holds true for $\beta = 1 $ thanks to the limit (2.8) in \cite{Hitchhiker}). \\
Next, we prove the continuity result for $\alpha = 1$. Let $(\alpha_k)_{k \in \mathbb{N}} \in ]\frac{1}{2}, 1[$ be a sequence such that ${\alpha}_k \underset{ k \rightarrow +\infty}{\longrightarrow} 1$. By replacing $\phi_{1,2,\mu, \beta}$ by the $\phi_{1}$ (with $\phi_{1}$ the principal eigenfunction which corresponds to the principal eigenvalue $\min(\xi_1(\Omega_1), \xi_2(\Omega_2)$), we deduce that 
\begin{equation}\label{xialphak}
0 < \xi_{\alpha_k}(\Omega_{1,2,\mu}) \leq  \mathrm{diam}(\Omega)^{2-2\alpha_k}\min(\xi_{1}(\Omega_1), \xi_2(\Omega_2)).
\end{equation}
We deduce that up to an extraction, $\xi_{\alpha_k}(\Omega_{1,2,\mu})$ converges to $\xi_\infty$. We recall that $\| \phi_{1,2,\mu, \alpha_k} \|_2 = 1$. Moreover, for all $k \in \mathbb{N}$, we have 
\begin{align*}
< (-\Delta)^{\alpha_k}\phi_{\alpha_k,1,2,\mu} , \phi_{\alpha_k,1,2,\mu} > &= \xi_{\alpha_k}(\Omega_{1,2,\mu}) \|\phi_{\alpha_k,1,2,\mu}\|_2^2\\
& \leq \mathrm{diam}(\Omega_{1,2,\mu})^{2(1-\alpha_k)}\min(\xi_1(\Omega_1), \xi_1(\Omega_2)).
\end{align*}
Up to an extraction and thanks to the Sobolev embedding (see \cite{Hitchhiker}), $\phi_{\alpha_k,1,2,\mu}$ converges to $\phi_\infty$ in $L^2(\Omega_{1,2, \mu})$. Moreover, the limit satisfies in a weak sense the following equation % Then, we can apply Theorem 1.2 of \cite{non_local_to_local} and we find that $\phi_\infty$ verifies in a weak sense the following equation
\begin{equation}\label{Phi_infty}
\left\lbrace 
\begin{aligned}
(-\Delta) \phi_\infty(x) &= \xi_\infty \phi_\infty(x) &&\text{ for } x \in \Omega_{1,2, \mu},  \\
\phi_\infty(x) &= 0 && \text{ for } x \in \partial \Omega_{1,2, \mu}.
\end{aligned}
\right.
\end{equation}
By the standard elliptic regularity, we find that \eqref{Phi_infty} is true in a strong sense. Furthermore, thanks to \eqref{xialphak}, we deduce that $\xi_\infty \leq \min(\xi_1(\Omega_1), \xi_2(\Omega_2))$. Since the $L^2$ norm of $\phi_\infty$ is not trivial and by positiveness of $\phi_\infty$ and uniqueness of $\phi_1$, we conclude that $\phi_\infty = \phi_{1}$. %Moreover, thanks to \eqref{phi_egal_1}, by compactness, we deduce the existence of $x_k \in \Omega_{1,2, \mu}$ such that $\phi_{\alpha_k,1,2,\mu}(x_k) = 1$. Up to an extraction, we can always assume that $x_k$ belongs to the same $\Omega_i$ (with $i \in \left\lbrace 1,2 \right\rbrace$) and converges to $x_\infty$ in $\overline{\Omega_i}$. Thus we deduce that $\underset{x \in \Omega_i}{\sup}\  \phi_\infty(x) = 1$ and $\phi_\infty \geq 0$ in $\Omega_i$. We deduce that (at least in $\Omega_i$), the function $\phi_\infty$ is a positive, non-trivial eigenfunction of the Laplacian. It follows that $\xi_\infty = \underline{\xi}_{i,1}$ and $\phi_\infty = \phi_i$ in $\Omega_i$ and $0$ outside, which concludes the proof of the theorem.

\end{proof}

% \begin{remark}
% We can go further and state that $\lambda_\alpha \underset{\alpha \rightarrow 1^-}{\longrightarrow} \lambda_\infty$ with $\lambda_\infty = \min( \underline{\lambda}_{1,1}, \underline{\lambda}_{2,1}) $.
% In other words, the limit eigenvalue is the eigenvalue of the larger domain. It is a consequence of \eqref{inegaliteLambda}. If $i \in \left\lbrace 1,2 \right\rbrace$ designates the larger domain (if the sizes are different) then we deduce that $\phi_\infty = 0$ in $\Omega_{3-i}$ (otherwise we have constructed a new eigenvalue in $\Omega_{3-i}$ which is less than the principal eigenvalue). Also, the element which realizes the maximum of $\phi_\infty$ belongs to $\Omega_i$.  
% \end{remark}

\subsection{Dependence on $\mu$ and $\alpha$}\label{subsection2-3}

For this subsection, we denote by $i \in \left\lbrace 1,2 \right\rbrace$ the index of the larger domain (i.e. $|\Omega_i| \geq |\Omega_{3-i}|$). This subsection is devoted to the 
\begin{proof}[Proof of Proposition \ref{Propo_alpha_mu}]
We split the proof of \eqref{lambdastar} in three parts: first, we assume $\xi_* = \xi_{1}(\Omega_{1,2,\mu})$, then we assume $\xi_* = \underline{\xi}_{i,1}$, and finally  we do the general case.\\
\textbf{Proof when $\xi_* = \xi_{1}(\Omega_{1,2,\mu})$. } Let $(\alpha_k)_{k \in \mathbb{N}}$ be any increasing sequence such that $\alpha_k \rightarrow 1$. According to Theorem \ref{theorem_3}, there exists $\underline{\mu_k}(\alpha_k)$ such that
\[\forall \mu \leq \underline{\mu_k}(\alpha_k), \ |\xi_{\alpha_k}(\Omega_{1,2,\mu }) - \xi_{\alpha_k}(\Omega_{1,2,0 })| \leq \frac{1}{k}.\]
We conclude that $(\alpha_k, \underline{\mu_k}(\alpha_k))$ gives the result. \\
\textbf{Proof when $\xi_* = \underline{\xi}_{i,1}$. } Let $(\nu_k)_{k \in \mathbb{N}}$ be any decreasing sequence such that $\nu_k \rightarrow 0$. According to Theorem \ref{theorem_4}, there exists $\overline{\beta_k}(\nu_k)<1$ such that
\[\forall \beta \geq \overline{\beta_k}(\nu_k), \ |\xi_\beta(\Omega_{1,2,\nu_k}) - \underline{\xi}_{i,1}| \leq \frac{1}{k}.\]
We conclude that $( \nu_k, \overline{\beta_k}(\nu_k))$ gives the result. \\
\textbf{Proof when $\xi_* \in ]\xi_1(\Omega_{1,2,0}), \underline{\xi}_{i,1}[$. } Let $(\mu_k, \alpha_k)$ be defined as follows: 
\begin{equation*}
\left\lbrace
\begin{aligned}
&(\mu_0, \alpha_0)=(\mu_1,\alpha_1)=(\mu_2,\alpha_2)=(\frac{1}{2}, \frac{1}{2}), \\
&(\mu_{3k}, \alpha_{3k})=(\mu_{3k-1}, \max(\overline{\beta_{k}}(\mu_{3k-1}), \alpha_{3k-1}, 1-\frac{1}{k})), \\
&(\mu_{3k+1}, \alpha_{3k+1}) = (\mu_{3k+1}, \alpha_{3k}), \\
&(\mu_{3k+2}, \alpha_{3k+2})=(\min(\underline{\mu_{k}}(\alpha_{3k}), \mu_{3k}, \frac{1}{k}), \alpha_{3k}),
\end{aligned}
\right.
\end{equation*}
where we will fix $\mu_{3k+1} \in [\mu_{3k}, \mu_{3k+2}]$ later on. The constants $\underline{\mu_{k}}(\alpha_{3k})$ and $ \overline{\beta_{k}}(\mu_{3k-1})$ are defined respectively in the two first parts of the proof. With a such choice of $(\mu_k, \alpha_k)$, we have 
\[\xi_{\alpha_{3k}}(\Omega_{1,2,\mu_{3k}}) \underset{ k \rightarrow +\infty}{\longrightarrow}  \underline{\xi}_{i, 1} \quad \text{ and } \quad \xi_{\alpha_{3k+2}}(\Omega_{1,2,\mu_{3k+2}}) \underset{ k \rightarrow +\infty}{\longrightarrow}  \xi_1(\Omega_{1,2,0}).\]
Let $k_0 \in \mathbb{N}$ be such that for all $k > k_0$, 
\begin{equation}\label{eq_int_prop_3}
\xi_1(\Omega_{1,2,0}) + \frac{2}{k} < \xi^* < \underline{\xi}_{i,1} - \frac{2}{k}.
\end{equation}
Inequalities \eqref{eq_int_prop_3} implies 
\[\xi_{\alpha_{3k}}(\Omega_{1,2,\mu_{3k}}) < \xi^* < \xi_{\alpha_{3k+2}} (\Omega_{1,2, \mu_{3k+2} }).\] 
Next, we fix $\mu_{3k+1} = \mu_{3k}$ for $k \in \left\lbrace 0, ..., k_0 \right\rbrace$. Since $\mu \mapsto \xi_{\alpha_k}(\Omega_{1,2,\mu})$ is continuous and increasing, and because $\alpha_{3k}=\alpha_{3k+1}$, we deduce thanks to the intermediate value theorem that there exists $\mu_{3k+1} \in ]\mu_{3k}, \mu_{3k+2}[$ such that $\xi_{\alpha_{3k+1}}(\Omega_{1,2,\mu_{3k+1}}) = \xi^*$. We conclude that the sequence $(\xi_k, \gamma_k) = (\mu_{3k+1}, \alpha_{3k+1})$ gives the result. 
\end{proof}

\section{Estimates for steady solutions}\label{section2}

This section is devoted to the proof of Theorem \ref{lemma1}. First, we provide the general strategy of the proof, next we prove intermediate results and finally, we prove Theorem \ref{lemma1}. We underline here that 
\[0 \leq n_+ \leq 1\] 
is easy to obtain thanks to the comparison principle. Therefore,we prove that there exists two constants $c,C>0$ such that, if the function $G(x)$ is given by \eqref{G1}, then we have
\[c \min(\delta(x)^\alpha, \varepsilon_0^\alpha) G(x) \leq n_+(x) \leq C \min(\delta(x)^\alpha, \varepsilon_0^\alpha) G(x) .\]

\noindent We highlight that the subdomains $\underline{\Omega}_{-,k}$ which satisfy \eqref{H6Chap4}  may not be connected. It is one of the interest to consider non-local diffusion instead of local one. Moreover, it may happen that the principal eigenvalue of $(-\Delta)^\alpha - Id$ defined in the connected components of $\underline{\Omega}_{-,k}$ are all positive however the principal eigenvalue of $(-\Delta)^\alpha - Id$ defined in the whole domain $\underline{\Omega}_{-,k}$ is negative. \\
A consequence of Theorem \ref{theorem_3} is the following: if $d=1$ and $\underline{\Omega}_- = \Omega_{1,2,\mu}$ (where $\underline{\Omega}_-$ refers to \eqref{H6Chap4} and $\Omega_{1,2,\mu}$ to \eqref{Omgega12mu}), the following assertions hold:
\begin{enumerate}
    \item we must have $\lambda_{\alpha}(\Omega_{1,2,\mu} ) < 0$, 
    \item if $\lambda_{\alpha}(\Omega_1)<0$ or $\lambda_{\alpha}(\Omega_2)<0$, then $\underline{\Omega}_- = \Omega_{1,2,\mu}$ holds true for all $\mu>0$,
    \item if $\lambda_{\alpha}(\Omega_1)>0$ and $\lambda_{\alpha}(\Omega_2)>0$, then there exists $\mu_0>0$ such that $\mu < \mu_0$.
\end{enumerate}
The first point is a consequence of the continuity up to $\mu = 0$ and the monotonicity of $(\mu \mapsto \lambda_\alpha (\Omega_{1,2,\mu}))$ established in Theorem \ref{theorem_3}. The second point is a direct consequence of Proposition \ref{theorem_1}. Finally the last point follows thanks to the continuity of the application $ \mu \mapsto \lambda_{\alpha} (\Omega_{1,2,\mu})$ and the following Lemma
\begin{lemma}\label{Lemma2}
Let $\Omega_{1,2,\mu}$ be a set as those introduced in \eqref{Omgega12mu}. If $\lambda_{\alpha}(\Omega_1), \lambda_\alpha (\Omega_2) > 0$, then there exists a distance $\mu_0 \geq 0$ such that $\lambda_\alpha(\Omega_{1,2,\mu_0}) = 0$. 
\end{lemma}
\noindent The proof of this lemma involves  lengthy but standard computations. Therefore, we postpone it at the end of the paper in the Appendix.% \ref{AppendixB1}.

\subsection{Strategy of the proof of Theorem \ref{lemma1}}

The lower part of \eqref{lemmaeq1} will be obtained  from the fractional Poisson kernel in a ball (see for instance in \cite{Bogdan2}) : 
\begin{equation}\label{PoissonKernel}
P_r(x,y)=c_\alpha^d \left( \frac{r^2-|x|^2}{|y|^2-r^2} \right)^\alpha \frac{1}{|x-y|^{d}} 1_{\left\lbrace |y|>r \right\rbrace }(x).
\end{equation}
With this kernel, for any smooth function $v$ and $z \in \mathbb{R}^d$ we have that the solution of the equation
\begin{equation*}
\left\lbrace
\begin{aligned}
(-\Delta)^\alpha \underline{n}(x) & = 0 && \quad x \in B( z, r), \\
\underline{n}(x) & = v(x) && \quad x \in B(z,r)^c, 
\end{aligned} 
\right. 
\qquad \text{ is } \qquad \underline{n}(x) = \int_{\mathbb{R}^d} P_r(x-z, y-z) v(y) dy. 
\end{equation*}

The difficult part of the proof will be to obtain the upper bound in $\Omega_+$ of \eqref{lemmaeq1}. First, we establish that:
\begin{lemma}\label{lemma2}
Let $n_+$ be a positive bounded solution of \eqref{stationary_state}. If $R_0>C_{d,\alpha, \varepsilon_0, \varepsilon_1, r_0}$ then there exists a positive constant $C_2$ such that for all $x \in \Omega_{+}$
\begin{equation}\label{lemmaeq2}
n_+(x) \leq C_2 \int_{\Omega_{-}} \frac{1}{|x-y|^{d+2\alpha}}dy.
\end{equation}
\end{lemma}

\noindent To prove Lemma \ref{lemma2}, we localise the function $n_+$ in a cluster $\mathcal{C}_k$ and we use hypothesis \eqref{H4Chap4} which essentially says  that the cluster is at large distance from the others.

\noindent Next, the idea is to compare the solution $n_+$ with a translated and rescaled barrier function $\psi$. This particular barrier function satisfies 
\begin{equation}\label{psi}
\left\lbrace
\begin{aligned}
&(-\Delta)^\alpha \psi (x) > 1  && \text{ for } x \in T(0, 1,4), \\
& \psi (x) = 0 && \text{ for } x \in B(0,1), \\
& \psi(x) \leq c(|x| - 1)^\alpha && \text{ for } x \in T(0,1,4), \\
& 1 \leq \psi \leq C && \text{ for } x \in B(0,4)^c ,
\end{aligned}
\right.
\end{equation}
where $T(z,r,R)$ designates the torus of center $z \in \mathbb{R}^d$ and of inner and outer radius $0<r<R$
\[\text{i.e. } \quad T(z,r,R) = \left\lbrace y \in \mathbb{R}^d, \ r<|y-z|<R \right\rbrace.\]
The construction of a such barrier function can be found in Appendix B of \cite{Ros-Oton-Serra}. Then, we look for "a suitable" constant $C_k$ such that for all $x \in \mathcal{C}_{+,k}$ we have 
\[n_+(x) \leq C_k \psi\left( \frac{x-z_x}{\varepsilon_0} \right)\]
where $z_x$ and $\varepsilon_0$ are introduced in \eqref{H2Chap4}. We claim that there exists a positive constant $C_{d,\alpha, \varepsilon_0, \varepsilon_1, r_0}'$ which depends only on the listed parameters such that 
\begin{equation}\label{ObjCk}
C_k =  \underset{ x \in \mathcal{C}_{+,k}}{\sup} \left( C_{d,\alpha, \varepsilon_0, \varepsilon_1, r_0}' \int_{\Omega_-} \frac{1}{|x-y|^{d+2\alpha}}dy  \right). 
\end{equation}
Thanks to \eqref{ObjCk}, the third property of \eqref{psi} and some Harnack type inequalities the conclusion follows. 
%Before going deeped in the technical details, we add some remark. 

\begin{remark}
In particular cases, some explicit and more tractable versions of $\mathcal{G}$ can be found. For instance, if $\Omega_-$ is bounded, one can prove that $\mathcal{G}(x)=\frac{\min(\delta(x)^\alpha, 1 )}{1+|x|^{d+2\alpha}}$. We provide some examples at the end of this section.
\end{remark}

\subsection{Harnack type properties}\label{appendixB}

First, we introduce the following function, defined for $x\in\Omega_+$
\begin{equation}\label{gchap4}
g(x)= \int_{\Omega_-} \frac{1}{|x-y|^{d+2\alpha}}dy,
\end{equation}
such that 
\[G(x) = 1_{\Omega_-}(x) + g(x) 1_{\Omega_+}(x).\]

\noindent We are going to prove some properties for $g$ and $G$ which are elementary in nature, but will be important for the sequel.

\begin{proposition}\label{propg1}
There exists a positive constant $C>0$ such that for all $k \in \mathbb{N}$, 
\begin{equation}\label{eqpropg1}
\underset{x \in \mathcal{C}_{+,k}}{\sup} g(x) \leq C \underset{x \in \mathcal{C}_{+,k}}{\inf} g(x).
\end{equation}
\end{proposition}
\begin{proof}
Let $k \in \mathbb{N}$ and $x_k \in \overline{\mathcal{C}_{+,k}}$ be such that
\[ \underset{x \in \mathcal{C}_{+,k}}{\sup} \int_{\Omega_-} \frac{1}{|x-y|^{d+2\alpha}}dy = \int_{\Omega_-} \frac{1}{|x_k-y|^{d+2\alpha}}dy.\]
According to \eqref{H4Chap4}, we have for all $(x,y)  \in \mathcal{C}_{+,k} \times \Omega_-$, 
\[ |x-y| \leq |x-x_k| + |x_k- y| \leq 2r_0 + |x_k-y|.\]
It follows that 
\[\frac{1}{(2r_0 + |x_k-y|)^{d+2\alpha}} \leq \frac{1}{|x-y|^{d+2\alpha}}.\]
Since $y \in \Omega_-$, we deduce that $R_0 \leq |x_k-y|$ and we conclude to the existence of a constant $C$ (independent of $x_k$ and $y$) such that 
\[\frac{1}{|x_k-y|} \leq \frac{C}{2r_0 + |x_k-y|}.\]
The conclusion follows. 
\end{proof}

\begin{corollary}\label{propg3}
There exists $\varepsilon_2>0$ such that for all $\varepsilon < \varepsilon_2$ we have 
\begin{equation}
\underset{|x-y|<\varepsilon_2}{\underset{x,y \in \Omega}{\inf}} \frac{G(x)}{G(y)}>0.
\end{equation}
\end{corollary}

With similar arguments, we prove 
\begin{lemma}\label{Hanack}
Let $\widetilde{g}: x \in \Omega_+ \mapsto  \int_{\Omega \cap B(z_x,R_0)^c} \frac{1}{|x-y|^{d+2\alpha}}dy$. There exists $C_{r_0,d,\alpha}$ depending on $r_0,  d, \alpha$ such that 
\[ \forall x \in \mathcal{C}_{+,k}, \quad \underset{y \in \mathcal{C}_{+,k}}{\max} \widetilde{g}(y) \leq C_{r_0,d,\alpha} \widetilde{g}(x).\]
\end{lemma}

\begin{proposition}\label{propg2}
The map $G$ is uniformly continuous in $\Omega_+$, that is
\[ \forall \varepsilon>0, \exists \mu_\varepsilon \text{ such that } \forall (x,y) \in \Omega_+, \ |x-y|<\mu_\varepsilon, \text{ then } |G(x) - G(y) | < \varepsilon. \]
\end{proposition}
\begin{proof} 
Let $\varepsilon > 0$ and $R> 0$ be such that 
\begin{equation}\label{proof21}
\int_{|y|>R} \frac{1}{|y|^{d+2\alpha}}dy = \frac{1}{R^{2\alpha}} \leq \frac{\varepsilon}{4}.
\end{equation}
Next, remarking that for all $(x,y) \in \Omega_+ \times \Omega_-$ we have $R_0 \leq |x-y|$, by the uniform continuity of the application $(t \in ]R_0, +\infty[ \mapsto t^{-(d+2\alpha)})$ we deduce the existence of a distance $\mu_\varepsilon>0$ such that for all $x_1,x_2 \in \Omega_+$ with $|x_1-x_2| \leq \mu_\varepsilon$ we have 
\begin{equation}\label{proof22}
 \forall y \in  \Omega_-, \ |\frac{1}{|x_1-y|^{d+2\alpha}} -  \frac{1}{|x_2 -y |^{d+2\alpha}} | \leq \frac{\varepsilon}{2 m(B(0, R + 2\mu_\varepsilon))}.
\end{equation}
Next, we conclude thanks to \eqref{proof21} and \eqref{proof22} that 
\begin{align*}
&|\int_{\Omega_-} \frac{1}{|x_1-y|^{d+2\alpha}}dy - \int_{\Omega_-} \frac{1}{|x_2-y|^{d+2\alpha}}dy | \\
&\leq |\int_{\Omega_- \cap B(x_1, R+2\mu_\varepsilon)} \frac{1}{|x_1-y|^{d+2\alpha}} - \frac{1}{|x_2-y|^{d+2\alpha}}dy | \\
& \quad  + |\int_{\Omega_- \cap B(x_1, R+2\mu_\varepsilon)^c} \frac{1}{|x_1-y|^{d+2\alpha}} dy|+| \int_{\Omega_- \cap B(x_1, R+\mu_\varepsilon)^c} \frac{1}{|x_2-y|^{d+2\alpha}}dy | \\
&\leq  \varepsilon. 
\end{align*}
\end{proof}

\noindent Finally, we recall the following strong maximum principle for the fractional Laplacian:

\begin{lemma}\label{maxprinc2}
 For any smooth bounded domain $\mathcal{U} \subset \mathbb{R}^d$, and any smooth non-trivial function $v \in H^\alpha(\mathcal{U})$ such that 
 \begin{enumerate}
     \item $v \geq 0$, 
     \item $((-\Delta)^\alpha - Id) v > 0$,
 \end{enumerate}
it follows $v > 0$ in the interior of the domain $\mathcal{U}$. 
\end{lemma}
%
%\begin{proof}
%Assume by contradiction that there exists $v$ which satisfies \textit{1.} and \textit{2.} and a point $x_0 \in \mathcal{U}$ such that $v(x_0) = 0$. It follows the following contradiction: 
%\[((-\Delta)^\alpha - Id )(v)(x_0) = \int_{\mathbb{R}^d} \frac{0 - v(y)}{|x-y|^{d+2\alpha}}dy <0.\]
%\end{proof}
%

\subsection{The proof of Lemma \ref{lemma2}}

\textbf{Step 1: Localisation argument. } Let $\chi_k$ be a function such that $\chi_k \in C^\infty_c (\mathbb{R}^d, [0,1])$ and 
\begin{equation}\label{chiR}
\chi_k(x,z) = 
\left\lbrace
\begin{aligned}
&1 && \text{ if } x \in B(z_k,r_0),\\
&0 && \text{ if } x \in B(z_k, 2r_0)^c
\end{aligned}
\right.
\end{equation}
where $z_k$ and $r_0$ are provided by \eqref{H4Chap4}. We set $n_k = n_+ \times \chi_k$. This function belongs to $H^\alpha_0(\mathcal{C}_{+,k})$ 
\begin{align*}
((-\Delta)^\alpha - Id) n_k(x) &= \int_{\mathbb{R}^d} \frac{n_k(x) - n_k (y)}{|x-y|^{d+2\alpha}}dy - n_k(x) \\
& = \int_{\mathbb{R}^d} \frac{n_+(x) - n_+ (y)}{|x-y|^{d+2\alpha}}dy + \int_{\mathbb{R}^d} \frac{ n_+ (y)(1-\chi_k(y))}{|x-y|^{d+2\alpha}}dy  - n_+(x) \\
&=\int_{\mathbb{R}^d} \frac{ n_+ (y)(1-\chi_k(y))}{|x-y|^{d+2\alpha}}dy - n_+(x)^2 \\
&\leq \int_{B(z_k, R_0)^c  \cap \Omega} \frac{n_+(y)}{|x-y|^{d+2\alpha}}dy. 
\end{align*}
In the following, we denote by 
\begin{equation}\label{vk}
\widetilde{g}_k = \underset{x \in \mathcal{C}_{+,k} }{\max} \int_{B(z_k, R_0)^c  \cap \Omega} \frac{n_+(y)}{|x-y|^{d+2\alpha}}dy
\end{equation}
Let $\phi_k$ be the principal eigenfunction of the Dirichlet operator $(-\Delta)^\alpha - Id $ in the set  
$$\left\lbrace x \in \mathbb{R}^d | \mathrm{dist}(x, \mathcal{C}_{+,k}) < \varepsilon_1 \right\rbrace.
$$
 Next, we define 
\begin{equation}\label{psik}
\psi_k(x) = \frac{\widetilde{g}_k}{c \lambda_0 \min(\varepsilon_1^{2\alpha}, 1)} \phi_k(x)
\end{equation}
where $c$ is a positive constant such that $\underset{y \in \mathcal{C}_{+,k}}{\min} \phi_k(y)>c\min(\varepsilon_1^{2\alpha}, 1)$. Indeed, since $\mathcal{C}_{+,k}$ is uniformly bounded with respect to $k$, we deduce thanks to \cite{Ros-Oton-Serra} (Proposition 3.5) and thanks to the interior ball condition that $c$ exists and is positive. Moreover, we have obviously
\[\widetilde{g}_k \leq ((-\Delta)^\alpha - Id) \psi_k(x) . \]
Since $n_+ \leq 1$, we deduce the existence of $C>1$ such that 
\[\forall x \in \mathcal{C}_{+,k}, \quad n_+(x) \leq C \psi_k(x).\]
Let $C^* =\inf  \left\lbrace C > 1 \ | \  n_+ \leq C \psi_k \right\rbrace$. We claim that $C^* = 1$. Indeed, if we assume by contradiction that $C^* \neq 1$, it follows that there exists $x_0 \in \overline{\mathcal{C}_{+,k}}$ such that $C^* \psi_k (x_0)- n_k(x_0) = 0  $. Recalling that ${n_{k}}_{| \partial \mathcal{C}_{+,k}} = 0$ and ${\psi_{k}}_{| \partial \mathcal{C}_{+,k}} \neq 0$, it follows that $x_0 \in \mathcal{C}_{+,k}$. Remarking that 
\[((-\Delta)^\alpha - Id )( C^* \psi_k - n_k)(x) \geq C^* \widetilde{g}_k - \widetilde{g} (x)\geq  0 ,\]
the existence of $x_0 \in \mathcal{C}_{+,k}$ is in contradiction with Lemma \ref{maxprinc2}. Therefore, we deduce that $C^* = 1$. We conclude thanks to Lemma \ref{Hanack} that for all $x\in \mathcal{C}_{+,k}$ 
\begin{equation}\label{step1Lemma1}
n_+(x) \leq n_k(x) \leq \psi_k(x) = \underset{ := C_0}{\underbrace{\frac{ C_{r_0,d,\alpha}}{c \lambda_0 \min(1, \varepsilon_1^{2\alpha} ) } } }  \ \widetilde{g}(x).
\end{equation}

\textbf{Step 2: Concentration in $\Omega_-$. }
Since $n_+\leq 1$, $R_0>C_0^{\frac{1}{2\alpha}}$ it follows by an immediate iteration of \eqref{step1Lemma1} that for all $x \in \Omega_+$: 
\[ n_+(x) \leq C_0 \overset{+\infty}{\underset{k=0}{\sum}}   I_k(x) \]
with 
\[I_k(x) =C_0 \int_{\Omega_+ \cap B(x,R_0)^c} \frac{I_{k-1}(y)}{|x-y|^{d+2\alpha}}dy, \qquad I_0 =  \int_{\Omega_-} \frac{1}{|x-y|^{d+2\alpha}}dy.\]
Let us prove by iteration that
\begin{equation}\label{objLemma1}
    I_k(x) \leq C_1^k \int_{\Omega_-} \frac{1}{|x-y|^{d+2\alpha}}dy \quad \text{ with } C_1 < 1.
\end{equation}
The conclusion that we may only use the components $\Omega_-$ will follows. It is clear that \eqref{objLemma1} holds true for $k=0$. Next, we prove it for $k=1$:
\begin{align*}
I_1(x) &= C_0 \int_{\Omega_{+} \cap B(z_x ,R_0)^c} \frac{1}{|x-y_0|^{d+2\alpha}} \ \int_{\Omega_-} \frac{1}{|y_0-y_1|^{d+2\alpha}} dy_1 dy_0 \\
&=C_0 \int_{(x-y'_0) \in \Omega_{+}, y'_0 \in B(0 ,R_0)^c} \frac{1}{|y'_0|^{d+2\alpha}} \ \int_{\Omega_-} \frac{1}{|x-y'_0-y_1|^{d+2\alpha}} dy_1 dy'_0.
\end{align*}
Next, by the convexity of the function $(t \mapsto |t|^{d+2\alpha} ) $ it follows 
\[ |x-y_1| \leq |x-y_0'-y_1|+|y_0'| \Rightarrow \frac{1}{|x-y'_0-y_1|^{d+2\alpha}} \leq \frac{2^{d-1+2\alpha}}{|x-y_1|^{d+2\alpha}} \left( 1 + \frac{|y_0'|^{d+2\alpha}}{|x-y'_0-y_1|^{d+2\alpha}} \right).\]
Fubini's Theorem leads to
\begin{align*}
I_1 (x) &\leq C_0 \int_{\Omega_-} \frac{1}{|x-y_1|^{d+2\alpha}} \int_{(x-y'_0) \in \Omega_{+}, y'_0 \in B(0 ,R_0)^c} \frac{2^{d-1+2\alpha}}{|y'_0|^{d+2\alpha}} \left(1 + \frac{|y_0'|^{d+2\alpha}}{|x-y'_0-y_1|^{d+2\alpha}} \right) dy_0' dy_1 \\
&\leq C_0 \int_{\Omega_-} \frac{ 2^{d-1+2\alpha}}{|x-y_1|^{d+2\alpha}} \left( \int_{ y'_0 \in B(0 ,R_0)^c} \frac{1}{|y_0'|^{d+2\alpha}}dy_0' + \int_{(x-y_0') \in \Omega_+}  \frac{1}{|x-y'_0-y_1|^{d+2\alpha}} dy_0' \right)  dy_1\\
&\leq  \frac{2^{d+2\alpha}C_0}{R_0^{2\alpha}}  \times  \int_{\Omega_-} \frac{1}{|x-y_1|^{d+2\alpha}}dy_1.
\end{align*} 
Assuming $R_0 > (2^{d+2\alpha} C_0)^{\frac{1}{2\alpha}}$ leads to the conclusion that \eqref{objLemma1} holds true for $k=1$.\\
If we assume that it is true for  $k\in \mathbb{N}$, we have by the recursive hypothesis on $I_k$ that
\[ I_{k+1}(x) \leq C_0  \int_{\Omega_+ \cap B(x,R_0)^c}  \frac{ C_1^k}{|x-y_0|^{d+2\alpha}} \int_{\Omega_-} \frac{1}{|y_0-y_1|^{d+2\alpha}} dy_1 dy_0.\]
The conclusions follows from  computations that are similar to those of the case $k=1$. 

\subsection{Proof of Theorem \ref{lemma1}}

\noindent Before providing the proof of Theorem \ref{lemma1} we introduce an intermediate technical result:

\begin{lemma}\label{lemma3}
There exist two constant $C, \sigma>0$ such that for all $x \in \Omega_+$, 
\begin{equation*}\label{lemmaeq3}
\begin{aligned}
&\int_{\Omega_-} \frac{1}{|x-y|^{d+2\alpha}}dy \leq C \int_{(\Omega_{-})_ \sigma} \frac{1}{|x-y|^{d+2\alpha}}dy\\
\text{and} \qquad  &   \int_{\mathcal{C}_{+,k_x}} \frac{1}{|x-y|^{d+2\alpha}}dy \leq C \int_{(\mathcal{C}_{+,k_x})_\sigma} \frac{1}{|x-y|^{d+2\alpha}}dy,
\end{aligned}
\end{equation*}
where $\mathcal{O}_\sigma$ is defined in \eqref{sets} for any set $\mathcal{O}$.
\end{lemma} 

\begin{proof}
We prove only the first inequality. The second one can be proved following similar computations. Thanks to the interior ball condition \eqref{H2Chap4}, it is sufficient to prove that  there exists $\sigma,C>0$ such that there holds 
\begin{equation}\label{goalLemma3} 
\int_{B(z, \varepsilon_0)} \frac{1}{ |x-y|^{d+2\alpha}} dy \leq C \int_{B(z, \varepsilon_0)_{\sigma}} \frac{1}{ |x-y|^{d+2\alpha}} dy 
\end{equation}
for any $x\in \Omega_+$ and $B(z, \varepsilon) \subset \Omega_-$. First, we remark that there exists two constants $c_1,c_2>0$ such that for all $t \in ]R_0, +\infty[$, 
\[ \frac{c_1}{t^{d+2\alpha}} \leq \frac{1}{(t + \varepsilon_0 )^{d+2\alpha}} \quad \text{ and }  \frac{1}{(t - \varepsilon_0 )^{d+2\alpha}} \leq \frac{c_2}{t^{d+2\alpha}}.\]
Next, denoting by $\mathrm{m}(E)$ the Lebesgue measure of the set $E$, we define 
\[\sigma = \frac{\varepsilon_0}{4} \quad  \text{ and } \quad C = \frac{2c_2\mathrm{m}(B(0, \varepsilon_0))}{c_1 \mathrm{m}(B(0, \frac{\varepsilon_0}{2})) }.\]
With a such choice of constants, it follows that 
\begin{align*}
C \int_{B(z, \varepsilon_0)_{\sigma}} \frac{1}{ |x-y|^{d+2\alpha}} dy  - \int_{B(z, \varepsilon_0)} \frac{1}{ |x-y|^{d+2\alpha}} dy &\geq \frac{C \mathrm{m}(B(0, \frac{\varepsilon_0}{2})) }{(|x-z| + \varepsilon_0)^{d+2\alpha}} - \frac{\mathrm{m}(B(0, \varepsilon_0))}{(|x-z|-\varepsilon_0)^{d+2\alpha}} \\
& \geq\frac{C c_1 \mathrm{m}(B(0, \frac{\varepsilon_0}{2})) }{|x-z| ^{d+2\alpha}} - \frac{c_2\mathrm{m}(B(0, \varepsilon_0))}{|x-z|^{d+2\alpha}} \\
& \geq \frac{ c_2 \mathrm{m}(B(0, \varepsilon_0))}{|x-z|^{d+2\alpha}} \geq 0.
\end{align*}

\end{proof}

\begin{proof}[Proof of Theorem \ref{lemma1}]
 We split the study into two parts: the study in $\Omega_+$ and the study in $\Omega_-$. Each part is split into two sub-parts :the lower and the upper bounds. 

\textbf{Part 1 : The study in $\Omega_-$. } \textbf{Subpart A : The lower bound. } Let $x \in \Omega_-$ and $z_x \in \Omega_-$, $k \in \left\lbrace 1,...,N \right\rbrace$ such that \eqref{H6Chap4} holds true. Since $\lambda_\alpha (\underline{\Omega}_k)<0$, we deduce the existence of $\underline{n}_{k,z}$ the solution of 
\begin{equation*}
    \left\lbrace 
    \begin{aligned}
    & (-\Delta)^\alpha \underline{n}_{k,z} = \underline{n}_{k,z_x} - \underline{n}_{k,z_x}^2 && \text{ in } \underline{\Omega}_k + z_x, \\
    &\underline{n}_{k,z_x} = 0 && \text{ in } (\underline{\Omega}_k + z_x)^c.
    \end{aligned}
    \right.
\end{equation*}
The maximum principle implies that 
\[ \underline{n}_{k,z_x}(x) \leq n_+(x).\]
Moreover, since $\underline{\Omega}_k$ is bounded and regular, we deduce the existence of $c_k$ such that 
\[c_k \min (\delta(x)^\alpha, 1) \leq \underline{n}_{k,z_x}(x).\]
If we fix $c=\underset{k \in \left\lbrace 1,...,N \right\rbrace}{\min} c_k$ then the conclusion follows. 
\bigbreak
\noindent \textbf{Subpart B : The upper bound. } From the maximum principle, it is clear that $n_+(x) \leq 1$. Therefore, we focus on what happens at the boundary: let $x \in \Omega_-$ such that $\delta(x) < \varepsilon_0$ and $\widetilde{z}_x \in \Omega^c$ provided by the exterior ball condition such that 
\[B(\widetilde{z}_x, \varepsilon_0) \cap \Omega = \emptyset, \quad \overline{B}(\widetilde{z}_x, \varepsilon_0) \cap \overline{\Omega} \neq \emptyset \quad \text{ and } \quad x \in B(\widetilde{z}_x,2 \varepsilon_0) .\]
Then, from the maximum principle applied to $n_+$ and $\psi(\frac{\cdot-\widetilde{z}_x}{\varepsilon_0})$ where $\psi$ is defined by \eqref{psi}, it follows 
\[n_+(x) \leq \psi(\frac{x-\widetilde{z}_x}{\varepsilon_0}) \leq C\delta(x)^\alpha.\]
The conclusion follows.

\textbf{Part 2: The study in $\Omega_+$. } \textbf{Subpart A : The lower bound. } We prove that for all $x \in \Omega_+$
\begin{equation}\label{lowerbound1}
c\min( \delta(x)^\alpha , \varepsilon_0^\alpha) \int_{\Omega_-}\frac{1}{|x-y|^{d+2\alpha}}dy \leq n_+(x).
\end{equation}
Let $x \in \Omega_+$ and $z_x \in \Omega_+$ provided by \eqref{H2Chap4} (remark that $z_x = x$ if $\delta(x) \geq \varepsilon_0$). We define $\underline{n}$ as the solution of 
\begin{equation*}
\left\lbrace
\begin{aligned}
(-\Delta)^\alpha \underline{n} (y) &= 0 && \text{ for } y \in B(z_x, \varepsilon_0), \\
\underline{n}(y) &= n_+ && \text{ for } y \in \Omega_-, \\
\underline{n}(y) &= 0 && \text{ for } y \in \mathbb{R}^d \backslash(B(z_x, \varepsilon_0) \cup \Omega_-).
\end{aligned}
\right.
\end{equation*}
The comparison principle gives that
\[\underline{n}(x) \leq n_+(x).\]
We recall that thanks to \eqref{H6Chap4} for all positive $\sigma$ small enough, there exists $c_\sigma>0$ such that $c_\sigma < n_+(y)$ for all $y \in (\Omega_-)_\sigma$. Formula \eqref{PoissonKernel} gives 
\begin{align*}
\underline{n}(x) &= \int_{\Omega_-} c_\alpha^d \left( \frac{\varepsilon_0^2-|x-z_x|^2}{|y-z_x|^2-\varepsilon_0^2} \right)^\alpha \frac{1}{|x-y|^{d}} n_+(y) dy \\ 
& \geq c c_\alpha^d (\varepsilon_0^2-|x-z_x|^2)^\alpha \int_{\Omega_{-, \sigma}}\frac{c_\sigma}{|y-z_x|^{d+2\alpha}}dy.
\end{align*}
If $ \varepsilon_0 \leq \delta(x)$, we deduce that $z_x =x$ and \eqref{lowerbound1} holds true thanks to Lemma \ref{lemma3}. Otherwise, we have
\[(\varepsilon_0 \delta(x) )^\alpha \leq (\varepsilon_0+ |x-z_x|)^\alpha \delta(x)^\alpha =  (\varepsilon_0^2-|x-z_x|^2)^\alpha  .\]
By uniform continuity and compactness of $B(z_x, \varepsilon_0)$, we deduce that for all $y \in \Omega_-$ there holds $\frac{c}{|x-y|^{d+2\alpha}} \leq \frac{1}{|z_x -y|^{d+2\alpha}}$, thus \eqref{lowerbound1} follows from Lemma \ref{lemma3}.\\
\bigbreak
\noindent \textbf{Subpart B: The upper bound.  } Thanks to Lemma \ref{lemma2}, we only have to consider $x \in \left\lbrace y \in \Omega_+ \text{ such that } \delta(y)< \varepsilon_0 \right\rbrace$. Let $k \in  \mathbb{N}$ and $x \in \mathcal{C}_{+,k}$ be such that $\delta(x) < \varepsilon_0$. Assumption \eqref{H2Chap4} ensures the existence of $\widetilde{z}_x \in \mathbb{R}^d$ such that 
\begin{equation}\label{zxchap4}
 B(\widetilde{z_x}, \varepsilon_0) \cap \Omega = \emptyset, \ \partial B(\widetilde{z_x}, \varepsilon_0) \cap \partial\Omega \neq \emptyset \text{ and } x\in  B(\widetilde{z_x}, 2\varepsilon_0) \cap \Omega . 
\end{equation}
As mentioned in section 2.1, the aim of the proof is to prove that the constant $C_k$ defined by \eqref{ObjCk} verifies 
\begin{equation}\label{aimlemma1up}
n_+(x) \leq C_k \psi \left(\frac{x - \widetilde{z}_x}{\varepsilon_0} \right)
\end{equation}
where $\psi_k$ is defined by \eqref{psi}. Let $\widetilde{\chi}_k$ be a function of $C^\infty_c (\mathbb{R}^d, [0,1])$ such that 
\begin{equation}\label{chik}
\widetilde{\chi}_k(x) = 
\left\lbrace
\begin{aligned}
&1 && \text{ for } x \text{ such that } \mathrm{dist}(x, \mathcal{C}_{+,k}) < \frac{\varepsilon_0}{4}\\
&0 && \text{ for } x \text{ such that } \mathrm{dist}(x, \mathcal{C}_{+,k}) > \frac{\varepsilon_0}{2}. 
\end{aligned}
\right.
\end{equation}
In order to prove \eqref{aimlemma1up}, we prove first that for all $y \in T(\widetilde{z}_x, \varepsilon_0, 4\varepsilon_0) \cap \mathcal{C}_{+,k}$ 
\begin{equation}\label{upboundgoal1}
(-\Delta)^\alpha \left(n_+\widetilde{\chi}_k(y) - C_k\psi(\frac{y-\widetilde{z}_x}{\varepsilon_0}) \right) < 0.
\end{equation}
Next, we prove that for all $y \in \left(T(\widetilde{z}_x, \varepsilon_0, 4\varepsilon_0) \cap \mathcal{C}_{+,k} \right)^c$, we have 
\begin{equation}\label{upboundgoal2}
n_+\widetilde{\chi}_k(y) - C_k\psi(\frac{y-\widetilde{z}_x}{\varepsilon_0}) \leq 0.
\end{equation}
The conclusion follows thanks to the maximum principle. 

\textbf{Proof that \eqref{upboundgoal1} holds true. } Let $y\in T(\widetilde{z}_x, \varepsilon_0, 4\varepsilon_0) \cap \mathcal{C}_{+,k}$, then thanks to the properties of $\psi$ (see \eqref{psi}), and Lemma \ref{lemma2}, we obtain
\begin{align*}
&(-\Delta)^\alpha \left(n_+\widetilde{\chi}_k(y) - C_k\psi(\frac{y-\widetilde{z}_x}{\varepsilon_0}) \right)\\
\leq& (-\Delta)^\alpha n_+(y) + \int_{\mathcal{C}_{+,k}^c} \frac{n_+(y_0)(1-\widetilde{\chi}_k(y_0))}{|y-y_0|^{d+2\alpha}}dy_0 - \frac{C_k}{\varepsilon_0^{2\alpha}}\\
\leq& n_+(y)   + \int_{\Omega_-} \frac{1}{|y_0-y|^{d+2\alpha}}dy_0+ \int_{\mathcal{C}_{+,k}^c \cap \Omega_+} \frac{n_+(y_0)}{|y-y_0|^{d+2\alpha}}dy_0 -C_k\\
 \leq& (C_2+1) \int_{\Omega_-} \frac{1}{|y-y_0|^{d+2\alpha}}dy_0 + C_2 \int_{\mathcal{C}_{+,k}^c \cap \Omega_+} \frac{1}{|y-y_0|^{d+2\alpha}} \int_{\Omega_-}\frac{1}{|y_0-y_1|^{d+2\alpha}}dy_1 dy_0 - C_k \\
\leq& \left( C_2 (1+\frac{C_{d, \alpha, r_0}}{R_0^{2\alpha}} ) + 1 \right) \int_{\Omega_-}\frac{1}{|y-y_0|^{d+2\alpha}}dy_0 - C_k.
\end{align*}
We conclude thanks to \eqref{ObjCk} that \eqref{upboundgoal1} holds true. 

\textbf{Proof that \eqref{upboundgoal2} holds true. } According to \eqref{zxchap4}, it is straightforward that $(n_+ \widetilde{\chi}_k(y) - C_k\psi(\frac{y-\widetilde{z}_x}{\varepsilon}))\leq 0$ for all $y \in B(\widetilde{z}_x, \varepsilon_0) \cup \left(B(\widetilde{z}_x , 4\varepsilon_0)^c \cap \mathcal{C}_{+,k}^c \right)$. Therefore, we focus on $y \in B(\widetilde{z}_x , 4\varepsilon_0)^c \cap \mathcal{C}_{+,k}$. Thanks to Lemma \ref{lemma2} and the properties of $\psi$ (introduced in \eqref{psi}), we obtain
\begin{equation*}
n_+(y) \chi_k (y) - C_k \psi(\frac{y - \widetilde{z}_x}{\varepsilon_0}) \leq n_+(y) - C_k \leq C_2 \int_{\Omega_-} \frac{1}{|y-y_0|^{d+2\alpha}}dy_0 - C_k < 0.
\end{equation*}
We deduce that \eqref{upboundgoal2} holds true. \\

The conclusion follows thanks to Proposition \ref{propg1}.
\end{proof}

\subsection{Some examples where $\mathcal{G}$ is explicit}

In this subsection, we detail some examples where the function $\mathcal{G}$ is more explicit. It highlights how the shape of the steady solution is strongly connected to the domain $\Omega$. 

\begin{enumerate}%[leftmargin=1cm]
    \item \textbf{$\Omega = \mathbb{R}^d$. } It is well known that in this case 
    \[\mathcal{G}(x) = 1.\]
    \item \textbf{$\Omega$ is bounded. } In this case, we have \[\mathcal{G}(x) = \delta(x)^\alpha.\]
    We recover here a consequence of \cite{Ros-Oton-Serra}.
    \item \textbf{$\Omega$ is unbounded and $\Omega_+ = \emptyset$. } Following Theorem \ref{lemma1}, there holds
    \[\mathcal{G}(x) = \min(1, \delta(x)^\alpha).\]
    Remark that this case contains the periodic patchy domain (see \cite{papier2} for more details):
    \[\Omega  = \underset{k \in \mathbb{Z}^d}{\bigcup} (\Omega_0 + a_k)\]
    where $\Omega_0 +a_k = \left\lbrace x \in \mathbb{R}^d \ | \ x-a_k \in \Omega_0 \right\rbrace$, $\Omega_0$ bounded and such that $\lambda_\alpha (\Omega_0) <0$ and $a_k \in \mathbb{R}^d$ such that 
    \[a_{k+e_i} - a_k = a_{e_i} \quad \text{ where } e_i \text{ is the } i^{th} \text{ element of the canonical basis of } \mathbb{Z}^d.\]
    \item \textbf{$\Omega$ is unbounded with $\Omega_-$ bounded. } In this case, up to a translation, we can assume that $B(0, \varepsilon_0) \subset \Omega_- \subset B(0, R_1)$ (with $R_1>0$ large enough). Then it is easy to observe that there exists two constants $c,C>0$ such that 
    \[\frac{c \min(\delta(x)^\alpha , 1) }{1+|x|^{d+2\alpha}} \leq \mathcal{G} (x) \leq  \frac{C \min(\delta(x)^\alpha , 1) }{1+|x|^{d+2\alpha}}.\]
    \item \textbf{The dimension $d=1$ and $\Omega_- = \mathbb{R}^-$. } In this case, a straightforward computation leads to \[c\min(\delta(x)^\alpha , 1) \left( 1_{\Omega_-}(x) + \frac{ 1_{\Omega_+}(x) }{1+|x|^{2\alpha}}   \right) \leq \mathcal{G} (x) \leq C \min(\delta(x)^\alpha , 1) \left( 1_{\Omega_-}(x) + \frac{ 1_{\Omega_+}(x) }{1+|x|^{2\alpha}}   \right).\]
\end{enumerate}

Of course, the hypothesis on $\Omega$ allows more general domains. However, with these five examples, we can already observe how the behavior of the solution $n_+$ may change from one domain to another.

\section{Uniqueness and attractivity of the steady state}\label{sectionUniqueness}

\subsection{Uniqueness}

As mentioned before, the idea to prove uniqueness is to compare two solutions $u_+$ and $v_+$ thanks to Theorem \ref{lemma1}. Before providing the details of the proof, we make an easy but important remark:
\begin{equation}\label{rqG}
\forall x \in \mathbb{R}^d, \quad    \mathcal{G}(x) \leq 1.
\end{equation}
In the case where $\Omega_+ = \emptyset$, the proof would be very easy. Because of the possibly degeneracy of the solutions in $\Omega_+$, the proof is not just an adaptation of the case $\Omega_+ = \emptyset$. It is inspired by the strategy of the proof of the fractional Hopf Lemma provided by Grecco and Servadei in \cite{Hopf_lemma_Greco}.

\begin{proof}[Proof of Theorem \ref{theorem1Chap4}]

Let $u_+$ and $v_+$ be two bounded non trivial solutions of \eqref{stationary_state}. Thanks to Theorem \ref{lemma1}, there exists a constant $L>0$ such that 
\[u_+ \leq L v_+.\]
We define $l= \inf \left\lbrace L > 1 \text{ such that } u_+ \leq L v_+ \right\rbrace$. We are going to prove by contradiction that $l=1$ because if $l=1$ then $u_+ \leq v_+$ and with the same argument, $v_+ \leq u_+$ and then the conclusion follows. Thus, we assume $l>1$ by contradiction. Next we define 
\[w = \underset{ x \in \Omega}{\inf}  \ \frac{(lv_+ - u_+ )(x)}{\mathcal{G}(x)}.\]
Thanks to the definition of $l$, we deduce that $w=0$ (otherwise we can construct $l^*<1$ such that $u_+< l^*v_+$ holds true, see the proof of Theorem 1 in \cite{papier2} for more details). Let $(x_n)_{n \in \mathbb{N}} \in \Omega$ be a minimizing sequence 
\[ \text{ i.e. }\qquad  \frac{(lv_+ - u_+ )(x_n)}{\mathcal{G}(x_n)} \underset{ n \rightarrow +\infty}{\longrightarrow} 0.\]
In order to localise where we use the maximum principle or the fractional Hopf Lemma, we will use the help from the bubble function $\beta_{z,r}$:
\begin{equation}\label{beta}
\beta_{z,r}(x) = C_\alpha (r^2 - |x-z|^2)^\alpha 1_{B(z,r)}(x) \qquad  \text{ with } C_\alpha = \frac{\Gamma (\frac{d}{2})}{4^\alpha \Gamma(1+\alpha) \Gamma(\frac{d}{2}+\alpha)}.
\end{equation}
Indeed, $\beta_{z,r}$ is solution of the equation 
\begin{equation}
\left\lbrace 
\begin{aligned}
(-\Delta)^\alpha \beta_{z,r} (x) &= 1 && \text{ for } x \in B(z,r), \\
\beta_{z,r}(x) &= 0 && \text{ for } x \in B(z, r)^c.
\end{aligned}
\right.
\end{equation}
We distinguish 2 cases: (up to a subsequence) $\inf \delta(x_n)>\delta_0>0$ and $\delta(x_n) \underset{n \rightarrow +\infty}{\longrightarrow} 0$. \\
In the first case, we define $z_n = x_n$ and $r = \delta_0$, and in the second one, we define $r= \varepsilon_0$ and we use the interior ball condition (hypothesis \eqref{H2Chap4}) to deduce the existence of $z_n\in \Omega$ such that 
\[\delta(x_n) = \varepsilon_0 - |z_n -x_n|.\]
\begin{figure}[h!]
\begin{center}
\includegraphics[width=12cm,height=4cm]{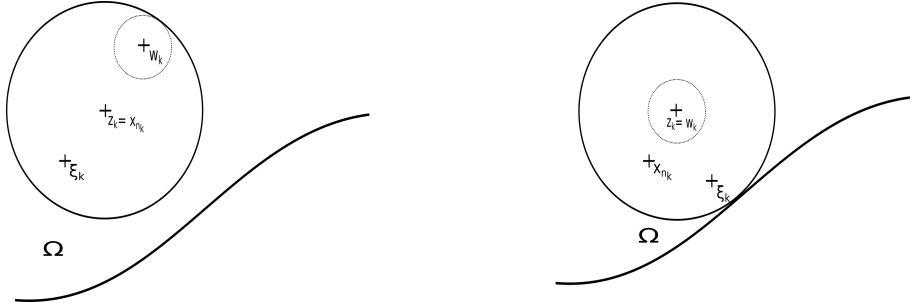}    
\caption{The left picture illustrates the case where $\inf \ \delta(\xi_k) > \delta_0$ (in this case, $z_k = x_k$). The right picture illustrates the case $\inf\delta(\xi_k) = 0$. In this case, up to a subsequence, we can assume that $\delta(\xi_k) < \frac{\varepsilon_0}{4}$ and $w_k = z_k$. }
\end{center}
\end{figure}

We prove by contradiction that there exists $\rho>0$ such that 
\begin{equation}\label{Obj5}
\forall x \in \mathbb{R}^d, \quad \rho \beta_{z_n, r}(x) \leq \frac{(lv_+ - u_+)(x_n)}{G(x_n)}.
\end{equation}
Indeed, if \eqref{Obj5} holds true for all $n \in \mathbb{N}$, then we deduce the following contradiction 
\begin{align*}
\rho C_\alpha r^\alpha = \underset{x \in B(0, \varepsilon_0)}{\inf} \frac{\rho \beta_{0, r}(x) }{(r-|x|)^\alpha} &= \underset{x \in B(z_n, r)}{\inf} \frac{\rho \beta_{z_n, r}(x) }{(r^\alpha-|x-z_n|)^\alpha} \\
&\leq \frac{(lv_+ - u_+)(x_n)}{\min(r_0,\delta(x_n))^\alpha  G(x_{n})} =  \frac{(lv_+ - u_+)(x_n)}{\mathcal{G}(x_{n})} \underset{n \rightarrow +\infty}{\longrightarrow} 0.
\end{align*}

Assume by contradiction that \eqref{Obj5} is false. Then since $\mathrm{supp}(\beta_{z_n, r}) = B(z_n, r)$, we deduce that 
\[\forall k \in \mathbb{N},\  \exists n_k>n_{k-1} \text{ and } y_k \in B(z_{n_k}, r) \text{ such that } \frac{(lv_+-u_+)(y_k)}{G(x_{n_k})} < \frac{\beta_{z_{n_k}, r} (y_k)}{k}.\]
Since for all $x \in \mathbb{R}^d \backslash B(z_{n_k}, r)$ we have 
\[ \frac{\beta_{z_{n_k}, r} (x)}{k} - \frac{(lv_+-u_+)(x)}{G(x_{n_k})} = -\frac{(lv_+-u_+)(x)}{G(x_{n_k})}  < 0\]
we deduce that $\frac{\beta_{z_{n_k},r}}{k} - \frac{(lv_+-u_+)}{G(x_{n_k})}$ takes its maximum at some $\xi_k \in B(z_{n_k}  , r)$. Remark that since $(\beta_{z_{n_k}, r})_{k \in \mathbb{N}}$ is uniformly bounded, we have that 
\begin{equation}\label{CV0}
0 < \frac{(lv_+-u_+)(\xi_k)}{G(x_{n_k})} < \frac{\sup \beta_{z_{n_k}, r}}{k} \underset{ k \rightarrow +\infty}{\longrightarrow} 0.
\end{equation}
Next, if we compute $(-\Delta)^\alpha (lv_+ -  u_+)$, we find that 
\begin{align*}
(-\Delta)^\alpha (lv_+ - u_+) &= lv_+ -lv_+^2-u_+ + u_+^2 \\
&= (lv_+ - u_+) - (lv_+ - u_+)(v_+ + u_+) + u_+v_+(l-1) \\
& \geq (lv_+ - u_+) (1-(u_+ + v_+)).
\end{align*}
Recalling that $u_+ \leq 1$ and $v_+ \leq 1$, we obtain 
\begin{equation}\label{Hopf1}
 -\frac{(lv_+-u_+)}{G(x_{n_k})} \leq  (-\Delta)^\alpha \frac{(lv_+ - u_+)}{G(x_{n_k})} .
\end{equation}
On one hand, if we evaluate \eqref{Hopf1} at $\xi_k$ we find thanks to \eqref{CV0} and Proposition \eqref{propg1} that for $k$ large enough 
\begin{equation}\label{HopfI0}
-\varepsilon \leq -\frac{C_1 G(\xi_k) \delta(\xi_k)^\alpha}{G(x_{n_k})} \leq -\frac{(lv_+-u_+)(\xi_k)}{G(x_{n_k})} \leq   (-\Delta)^\alpha \frac{(lv_+ - u_+)(\xi_k)}{G(x_{n_k})} 
\end{equation}
with $\varepsilon$ as small as we want. \\
On the other hand, we claim that there exists $c_*>0$ such that $   (-\Delta)^\alpha \frac{(lv_+ - u_+)(\xi_k)}{G(x_{n_k})} \leq -c_* + \varepsilon$ with $\varepsilon$ as small as we want. Indeed, we first introduce $w_k \in B(z_{n_k}, r)$ such that 
\begin{equation}\label{wk}
B(w_k, \frac{r}{4}) \subset B(z_{n_k}, r), \quad \delta(w_k) > \frac{r}{2} \quad \text{ and } \quad \forall y \in B(w_k, \frac{r}{4}), \ |y- \xi_k| > \frac{r}{4}
\end{equation}
(remark that one can take $w_k = z_{n_k}$ if $\delta( x_{n_k}) \underset{k \rightarrow +\infty}{\longrightarrow} 0$). Next, we split $(-\Delta)^\alpha \frac{(lv_+ - u_+)(\xi_k)}{G(x_{n_k})}$ in the following way
\begin{align*}
(-\Delta)^\alpha \frac{(lv_+ - u_+)(\xi_k)}{G(x_{n_k})} =& \frac{1}{G(x_{n_k})}\int_{B(w_k, \frac{r}{4})} \frac{(lv_+ - u_+)(\xi_k) - (lv_+ - u_+)(y)}{|\xi_k - y|^{d+2\alpha}}dy \\
& + \frac{1}{G(x_{n_k})} \int_{B(w_k, \frac{r}{4})^c} \frac{(lv_+ - u_+)(\xi_k) - (lv_+ - u_+)(y)}{|\xi_k - y|^{d+2\alpha}}dy \\
=& I_1 + I_2.
\end{align*}
For $I_1$, according to \eqref{CV0}, \eqref{wk}, Theorem \ref{lemma1} and Corollary \ref{propg3} we deduce that for $k$ large enough we have
\begin{equation}\label{HopfI1}
\begin{aligned}
I_1 &= \frac{m(B(0,  \frac{r}{4}))(lv_+ - u_+)(\xi_k)}{G(x_{n_k})} - \int_{B(w_k, \frac{r}{4})}  - \frac{(lv_+ - u_+)(y)}{G(x_{n_k})|\xi_k - y|^{d+2\alpha}}dy \\
& \leq \frac{m(B(0,  \frac{r}{4}))(lv_+ - u_+)(\xi_k)}{G(x_{n_k})} - \int_{B(w_k, \frac{r}{4})}   \frac{(c_1 + C_1) G(y) \min(\delta(y)^\alpha , 1) }{G(x_{n_k})|\xi_k - y|^{d+2\alpha}}dy \\
& \leq \frac{\varepsilon}{2} - r^\alpha (c_1 + C_1) \  \underset{y \in B(w_k, \frac{r}{4})}{\inf} \  \frac{G(y)}{G(x_{n_k})} \   \underset{ \xi \in T(0,\frac{r}{4}, \frac{3r}{4}) }{\inf} \  \int_{B(0, \frac{r}{4})} \frac{1}{|\xi-y|^{d+2\alpha}}dy\\
&=\frac{\varepsilon}{2} -c_*.
\end{aligned}
\end{equation}
Note that $c_*$ depends only on $r$ and not on $k$. For $I_2$, we recall that 
\[\underset{ x \in \mathbb{R}^d}{\max}\ \left(  \frac{\beta_{z_{n_k}, r} (x)}{k} - \frac{(lv_+-u_+)(x)}{G(x_{n_k})} \right) = \frac{\beta_{z_{n_k}, r}(\xi_k) }{k} - \frac{(lv_+-u_+) (\xi_k)}{G(x_{n_k})}. \]
Then we deduce that for $k$ large enough we have
\begin{equation}\label{HopfI2}
\begin{aligned}
I_2 &\leq \frac{1}{k} \int_{B(w_k, \frac{r}{4})^c} \frac{\beta_{z_{n_k}, \varepsilon_0} (\xi_k) - \beta_{z_{n_k}, \varepsilon_0} (y)}{|\xi_k - y  |^{d+2\alpha}}dy \\
&= \frac{1}{k} \left((-\Delta)^\alpha \beta_{z_{n_k}, r}(\xi_k) -  \int_{B(w_k, \frac{r}{4})} \frac{\beta_{z_{n_k}, r} (\xi_k) - \beta_{z_{n_k}, r} (y)}{|\xi_k - y  |^{d+2\alpha}}dy \right)\\
& \leq  \frac{1}{k} \left(1 + \underset{\xi \in T(0, \frac{r}{4}, \frac{3r}{4})}{\max} \int_{B(0, \frac{\varepsilon_0}{2})} \frac{\beta_{0, r} (y)}{|\xi - y  |^{d+2\alpha}}dy \right)\\
& \leq \frac{\varepsilon}{2}.
\end{aligned}
\end{equation}
Finally, combining \eqref{HopfI0}, \eqref{HopfI1} and \eqref{HopfI2}, we deduce that for $k$ large enough we have 
\[-\varepsilon \leq   (-\Delta)^\alpha \frac{(lv_+ - u_+)(\xi_k)}{G(x_{n_k})} \leq -c_* + \varepsilon\]
which is a contradiction for $\varepsilon = \frac{c_*}{4}$.

\end{proof}

\subsection{Convergence to the steady state}\label{section4}

\noindent The idea is quite classical: it consists in  enclosing the solution of \eqref{equation_with_time} between an increasing (with respect to the time) sub-solution and a decreasing (with respect to the time) super-solution. The specific ingredient is, once again,  that the behaviour of the solution of the Cauchy Problem at infinity has to match that of the steady solution. This is achieved through estimates of the heat kernel developed in \cite{papier2} (Theorem 2 p.3) for general domains which satisfy the uniform interior and exterior ball condition (here assumed in the assumption \eqref{H2Chap4}). The details being otherwise standard, we just present an overview of the proof. 

\bigbreak

\noindent \textbf{The super-solution. } Up to a translation, there is no loss of generality to assume that
\[ B(0, \varepsilon_0) \subset \Omega_-.\]
On one hand thanks to Lemma \ref{Lemma2}, there exists some constants $c>0$ (that may change from line to line) such that  
\begin{equation}\label{timesuper1}
\begin{aligned}
    n_+(x) &\geq c \min (\delta(x)^\alpha , 1) \times \left(1_{\Omega_-}(x) + \int_{\Omega_-} \frac{1_{\Omega_+}(x)}{|x-y|^{d+2\alpha}}dy \right) \\ 
    &\geq c \min (\delta(x)^\alpha , 1) \times \left(1_{\Omega_-}(x) + \int_{B(0,\varepsilon_0)} \frac{1_{\Omega_+}(x)}{|x-y|^{d+2\alpha}}dy \right) \\ 
    &\geq \frac{c \min (\delta(x)^\alpha , 1)}{1+|x|^{d+2\alpha}}. 
\end{aligned}
\end{equation}
On the other hand, thanks to Theorem 2 in \cite{papier2}, there exists a constant $C>0$ such that 
\begin{equation}\label{timesuper2}
     n(x,t=1) \leq  \frac{C \min(\delta(x)^{\alpha},1)}{1+|x|^{d+2\alpha}} . 
\end{equation}
From \eqref{timesuper1} and \eqref{timesuper2}, we deduce that there exists a constant $C>1$ such that 
\[n(x,t=1) \leq C n_+(x).\]
Next, we define 
\begin{equation*}
    \left\lbrace 
    \begin{aligned}
        &\partial_t \overline{n} + (-\Delta)^\alpha \overline{n} = \overline{n} - \overline{n}^2&& \quad \text{ for } (x,t) \in \Omega \times ]0, +\infty[, \\
        &\overline{n}(x,t)= 0&& \quad \text{ for } (x,t) \in \Omega^c \times [0, +\infty[, \\
        & \overline{ n} (x,t=0) = C n_+(x)&& \quad \text{ for } x \in \Omega.
    \end{aligned}
    \right.
\end{equation*}
Since, in the distributional sense, we have
\[\underset{ t \rightarrow 0^+}{\lim } \partial_t \overline{n}(x,t) =  \underset{ t \rightarrow 0^+}{\lim } (C\overline{n}(x,t) - C\overline{n}(x,t)^2 - (-\Delta)^\alpha C\overline{n}(x,t)) = C n_+(x) (1-C)\leq 0.\]
We deduce that $\overline{n}$ is decreasing. Since it is bounded from below by $n_+$, it converges to a non-trivial stationary state of \eqref{stationary_state} and by uniqueness of $n_+$ we deduce that $\overline{n}(x,t) \underset{t \rightarrow +\infty}{\longrightarrow} n_+(x)$. 

\bigbreak

\noindent \textbf{The sub-solution. } As previously, we assume that $ B(0, \varepsilon_0) \subset \underline{\Omega}_1 \subset \Omega_-$ where $\underline{\Omega}_1$ is defined by \eqref{H6Chap4}. Thanks to Theorem 2 in \cite{papier2}, there exists a constant $c>0$ such that 
\[\frac{c \min( \delta(x)^\alpha , 1)}{1+|x|^{d+2\alpha}} \leq n(x,t=1).\]
It follows that there exists $\sigma>0$ such that 
\[\sigma < \underset{x \in \underline{\Omega}_1}{\inf} n(x,t=1) .\]
Next, if we denote by $\phi_1$ the principal eigenfunction of $(-\Delta)^\alpha - Id $ in $\underline{\Omega}_1$, it follows that the solution $\underline{n}$ of 
\begin{equation*}
    \left\lbrace 
    \begin{aligned}
        &\partial_t \underline{n} + (-\Delta)^\alpha \underline{n} = \underline{n} - \underline{n}^2&& \quad \text{ for } (x,t) \in \Omega \times ]0, +\infty[, \\
        &\underline{n}(x,t)= 0&& \quad \text{ for } (x,t) \in \Omega^c \times [0, +\infty[, \\
        & \underline{ n} (x,t=0) = \frac{\min(\sigma, |\lambda_1|) \phi_1(x)}{\| \phi_1\|_\infty }&& \quad \text{ for } x \in \Omega.
    \end{aligned}
    \right.
\end{equation*}
is increasing with respect to the time. Indeed, it is sufficient to verify it at time $t=0$. For $x \in \underline{\Omega}_1$, there holds (in a distributional sense)
\begin{align*}
\underset{ t \rightarrow 0^+}{\lim} \  \partial_t \underline{n}(x,t) &= \underset{ t \rightarrow 0^+}{\lim} \   ( \underline{n}(x,t) -  \underline{n}(x,t)^2 -(-\Delta)^\alpha  \underline{n}(x,t) )\\
&= \frac{\min(\sigma, |\lambda_1|) \phi_1(x) }{\|\phi_1\|_\infty} ( |\lambda_1| - \frac{\min(\sigma, |\lambda_1|) \phi_1(x)}{\|\phi_1\|_\infty} )\\
&\geq 0.
\end{align*}
For $x \in \Omega \backslash \underline{\Omega}_1$, there holds (still in a distributional sense)
\[\underset{ t \rightarrow 0^+}{\lim} \ \partial_t \underline{n}(x,t) = \underset{ t \rightarrow 0^+}{\lim} \   ( \underline{n}(x,t) -  \underline{n}(x,t)^2 -(-\Delta)^\alpha  \underline{n}(x,t) ) = \int_{\mathbb{R}^d} \frac{ \frac{\min(\sigma, |\lambda_1|) \phi_1(y)}{\| \phi_1\|_\infty }}{|x-y|^{d+2\alpha}}dy > 0.\]
Finally, $\underline{n}$ is increasing and bounded therefore point-wise converging. By fractional elliptic regularity, the limit is a solution of \eqref{stationary_state} and by uniqueness of the non trivial stationary state $n_+$ we conclude that $\underline{n}(x,t) \underset{t \rightarrow +\infty}{\longrightarrow} n_+(x)$.

\bigbreak 

\textbf{Conclusion. } Since the initial datum are right ordered 
\[ \text{i.e. } \quad \underline{n}(x,t=0) \leq n(x,t=1) \leq \overline{n}(x,t=0),\]
we conclude thanks to the comparison principle and the conclusions of the two lasts parts that 
\[\underset{t \rightarrow +\infty}{\lim} \underline{n}(x,t) = n_+(x) \leq \underset{t \rightarrow +\infty}{\lim } n(x,t) \leq \underset{t \rightarrow +\infty}{\lim} \overline{n}(x,t) = n_+(x).\]
This ends the proof of Corollary \ref{propagation}.

\section{Numerical illustrations, perspectives}\label{section_numeric}

We provide some numerical illustrations of the results developed in this section. More precisely, we investigate the large time of simulations of equation \eqref{equation_with_time} with two 1-dimensional disconnected components modeled by a finite difference method. We vary the distance $\mu$ between the two components and the exponent of the diffusivity $\alpha$. We recover that if the distance $\mu$ is to high or the constant $\alpha$ is to closed to $1$ then the solution $n$ of \eqref{equation_with_time} tends to $0$ which means that $\lambda_{\alpha}(\Omega_{1,2,\mu})\geq 0$. Whereas if the distance $\mu$ is not to high and the constant $\alpha$ is not closed to $1$ then the solution $n$ tends to a non-trivial positive stationary state which means that $\lambda_{\alpha}(\Omega_{1,2,\mu}) < 0$.

The first simulation (Figure \ref{Fig1Chap4}) shows the numerical solution $n$ at time $T=2000$ and $T=4000$. We can not distinguish the difference between the two drawings, we deduce that we have reached the stationary state. 
\begin{figure}[h]
\caption{$\alpha = 0.5, \ \mu=0.5, \ \Omega_1=\Omega_2 = ]0,2[$}
\centering
\includegraphics[width=12cm,height=6cm]{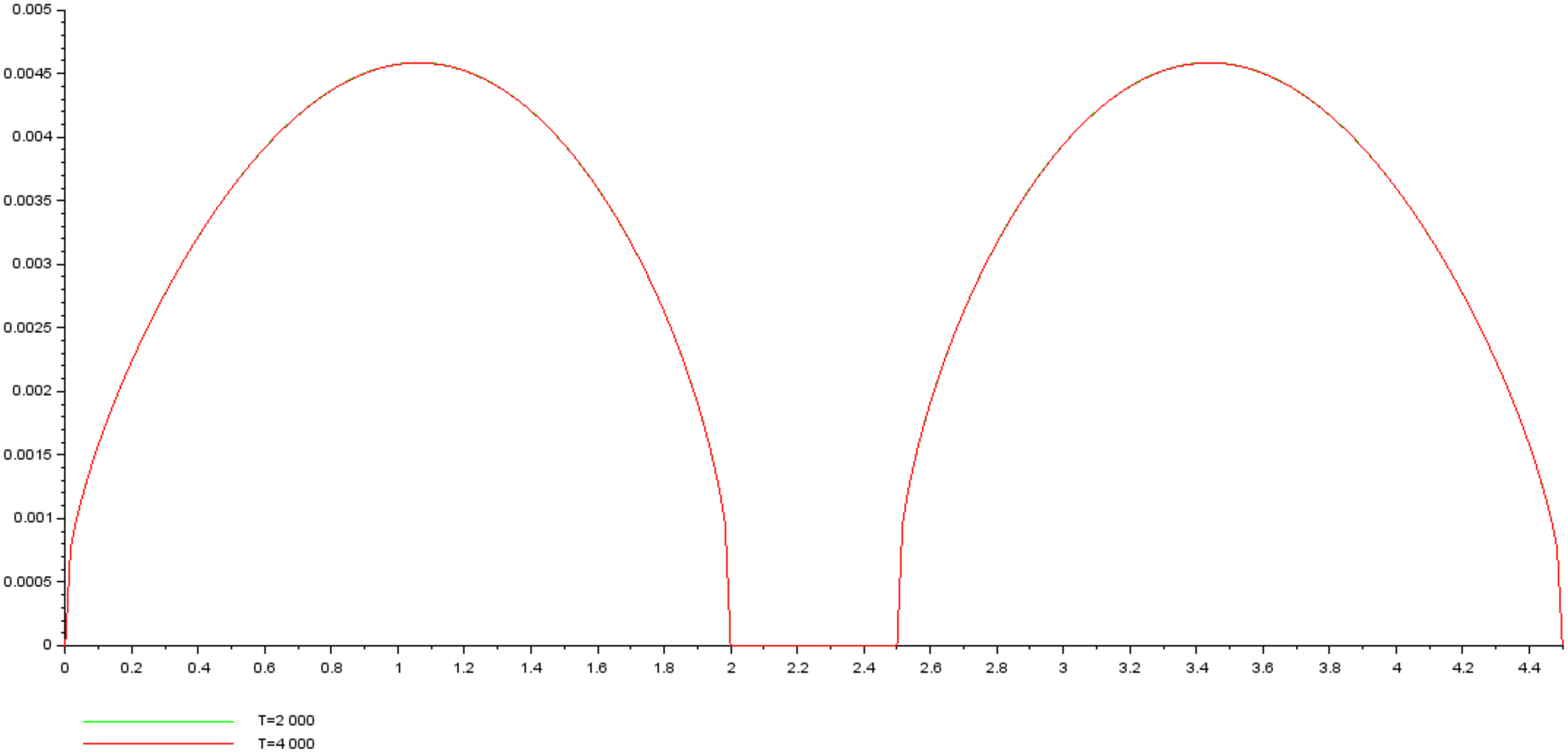} \label{Fig1Chap4}
\end{figure}

Next, in Figure \ref{Fig2} we increase $\alpha$. We put $\alpha = 0.8$, and we find that the solution is almost $0$. 
\begin{figure}[h]
\caption{$\alpha = 0.8, \ \mu=0.5, \ \Omega_1=\Omega_2 = ]0,2[$}
\centering
\includegraphics[width=12cm,height=6cm]{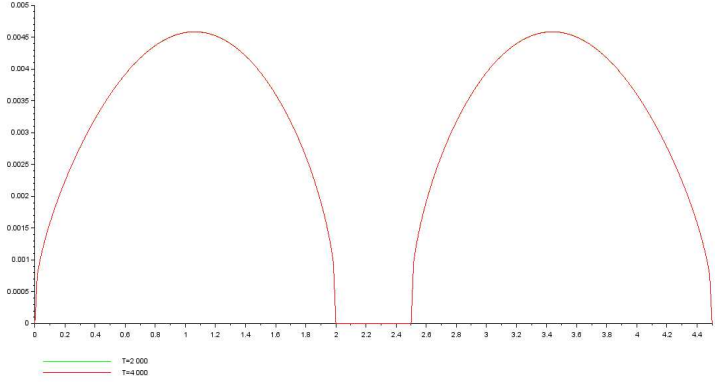} \label{Fig2}
\end{figure}
\newpage
Finally, Figure \ref{Fig3} illustrate Theorem \ref{theorem_4} by the following simulation :
\begin{figure}[h]
\caption{$\alpha = 0.5, \ \mu=0.0001, \ \Omega_1=\Omega_2 = ]0,2[$}
\centering
\includegraphics[width=12cm,height=6cm]{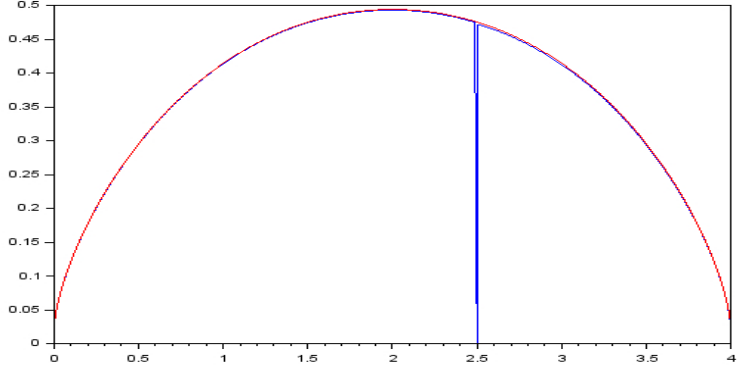} \label{Fig3}
\end{figure}

As a conclusion to this work, let us recall that, after establishing conditions on a domain $\Omega$ which ensure the existence and the uniqueness of the stationary state of the fractional Fisher-KPP equation, we focus on the principal eigenvalue of $(-\Delta)^\alpha - Id$ in one dimension of domain composed by two bounded connected components. This study is strongly related to the issue of existence and uniqueness of the stationary state of the fractional Fisher-KPP equation.

The perspective are the followings. We would like to relax the hypothesis on the domain $\Omega$. Indeed, rather than a minimal distance between two patches, we would like to assume that 
\[\exists \lambda_0>0 \text{ such that } \forall R>0,\quad  \lambda_\alpha (\Omega_+ \cap B(0, R) ) > \lambda_0.\]
We also expect to prove the continuity result on the principal eigenvalue in the multi-dimensional case. Finally, we would like to have a better understanding of the dynamic of $\lambda_{\alpha}(]-a-\mu, -\mu [\cup]\mu, a+\mu[)$ when $(\alpha ,\mu) \rightarrow (1,0)$.

% \vspace{1cm}

% \noindent \textbf{Acknowledgements.}  Both authors were supported by the ANR project NONLOCAL ANR-14-CE25-0013. They thank R. Aguilée, L. Caffarelli,  S. Mirrahimi and Y. Sire for many fruitful discussions that allowed  this work to make progress. They thank the university of  Texas at Austin, and specially L. Caffarelli for their hospitality. 

 %\appendix

\section*{Appendix - Proof of Lemma \ref{Lemma2}}\label{AppendixB1}
\addcontentsline{toc}{section}{Appendix - Proof of Lemma \ref{Lemma2}}

%
    %\pagestyle{fancy}
    %\renewcommand{\sectionmark}[1]{%
    %\markboth{\thesection.\ }\MakeUppercase{#1}{}}
    %\fancyhf {}
  %\makeatletter
  %\renewcommand{\chaptermark}[1]{\markboth{\bsc{\@chapapp~\thechapter{} :} #1}{}}
  %\makeatother
  %\renewcommand{\sectionmark}[1]{\markright{\thesection.{} #1}}
  %\renewcommand{\headrulewidth}{1pt}   %Epaisseur de la ligne.
  %\renewcommand{\footrulewidth}{0pt}   %Epaisseur de la ligne.
  %\lhead[\textsl{\leftmark}]{{}}
  %\rhead[ {} ]{ {\textsl{Appendix - Proof of Lemma \ref{Lemma2}}} }
  %\lfoot[\thepage]{{}}
  %\rfoot[{}]{\thepage}
  %\cfoot[{}]{{}}
%

We provide here the proof of Lemma \ref{Lemma2}. Before giving the proof,  we introduce a new notation: 
\begin{notation}
For all $\phi \in H^\alpha_0 (\Omega_{1,2,\mu})$ and $i \in \left\lbrace 1,2 \right\rbrace$, we will denote by $\phi^i$ the function $\phi$ restricted to the set $\Omega_i$ and extended by $0$ outside $\Omega_i$
\[\text{i.e.  } \ \phi^i (x) = \phi(x) 1_{\Omega_i}(x).\]
We also denote by $\underline{\lambda}_i$ the principal eigenvalue of $(-\Delta)^\alpha - Id$ in $\Omega_i$ with $0$ exterior Dirichlet conditions. 
\end{notation}
\begin{remark}
For all $\phi \in H^\alpha_0(\Omega_{1,2,\mu})$, we have
\[\phi(x) = \phi^1(x) + \phi^2 (x) \ \text{ and } \ \phi^i \in H^\alpha_0(\Omega_i) \  \text{ for all } \ i \in \left\lbrace 1,2 \right\rbrace. \]
\end{remark}

\begin{proof}
The aim of the proof is to prove that for $\mu$ large enough, there exists $C>0$ such that 
\[\lambda_\alpha (\Omega_{1,2,\mu }) \geq \min(\underline{\lambda}_{1 ,\alpha}, \underline{\lambda}_{2, \alpha}) - \frac{C}{\mu^{1+2\alpha}}.\]
The conclusion follows by the intermediate value Theorem. We start from the Rayleigh quotient defining $\underline{\lambda}_{i, \alpha}$: 

\begin{align*}
\underline{\lambda}_{i,\alpha} &=  \frac{ \int_{\mathbb{R}}\int_\mathbb{R} \frac{(\underline{\phi}_{i,\alpha}(x) - \underline{\phi}_{i,\alpha}(y))^2}{|x-y|^{1+2\alpha}}dy - (\underline{\phi}_{i,\alpha}(x))^2 dx }{\int_{\Omega_i} (\underline{\phi}_{i,\alpha}(x))^2dx}  \\
&=  \frac{ \int_{\Omega_i}\int_\mathbb{R} \frac{(\underline{\phi}_{i,\alpha}(x) - \underline{\phi}_{i,\alpha}(y))^2}{|x-y|^{1+2\alpha}}dy - (\underline{\phi}_{i,\alpha}(x))^2 dx }{\int_{\Omega_i} (\underline{\phi}_{i,\alpha}(x))^2 dx} +  \frac{ \int_{\mathbb{R} \backslash \Omega_i}\int_{\Omega_i} \frac{(\underline{\phi}_{i,\alpha}(y))^2}{|x-y|^{1+2\alpha}}dy  dx }{\int_{\Omega_i} (\underline{\phi}_{i,\alpha}(x))^2 dx} \\
&= \frac{ \int_{\Omega_i}\int_{\Omega_i} \frac{(\underline{\phi}_{i,\alpha}(x) - \underline{\phi}_{i,\alpha}(y))^2}{|x-y|^{1+2\alpha}}dy - (\underline{\phi}_{i,\alpha}(x))^2 dx }{\int_{\Omega_i} (\underline{\phi}_{i,\alpha}(x)) ^2 dx} + 2 \frac{ \int_{\mathbb{R} \backslash \Omega_i}\int_{\Omega_i} \frac{(\underline{\phi}_{i,\alpha}(y))^2}{|x-y|^{1+2\alpha}}dy  dx }{\int_{\Omega_i} (\underline{\phi}_{i,\alpha}(x))^2 dx}.
\end{align*}
We continue in the same way by rewriting $\lambda_{1,2,\mu,\alpha}$:
\begin{align*}
\lambda_\alpha (\Omega_{1,2,\mu }) =&  \frac{ \int_{\mathbb{R}}\int_\mathbb{R} \frac{(\phi_{1,2,\mu,\alpha}(x) - \phi_{1,2,\mu,\alpha}(y))^2}{|x-y|^{1+2\alpha}}dy - (\phi_{1,2,\mu,\alpha}(x))^2 dx }{\int_{\Omega_{1,2,\mu}} (\phi_{1,2,\mu,\alpha}(x))^2dx} \\
=&  \frac{\int_{\Omega_{1,2,\mu}}\int_\mathbb{R} \frac{(\phi_{1,2,\mu,\alpha}(x) -\phi_{1,2,\mu,\alpha}(y))^2}{|x-y|^{1+2\alpha}}dy - (\phi_{1,2,\mu,\alpha}(x))^2 dx }{\int_{\Omega_{1,2,\mu}} (\phi_{1,2,\mu,\alpha}(x))^2dx}\\
&+ \frac{ \int_{\mathbb{R} \backslash \Omega_{1,2,\mu}}\int_\mathbb{R} \frac{(\phi_{1,2,\mu,\alpha}(x) - \phi_{1,2,\mu,\alpha}(y))^2}{|x-y|^{1+2\alpha}}dydx}{\int_{\Omega_{1,2,\mu}} (\phi_{1,2,\mu,\alpha}(x))^2dx}.
\end{align*}
Thus, we have found:
\begin{equation}\label{pf_thm_2_lambda}
\lambda_\alpha (\Omega_{1,2,\mu })=  \frac{1}{\int_{\Omega_{1,2,\mu}} (\phi_{1,2,\mu,\alpha}(x))^2dx} \left( I_1+I_2 \right).
\end{equation}
We rewrite $I_1$ and $I_2$ in order to involving the expression of $\underline{\lambda}_{1,\alpha}$ and $\underline{\lambda}_{2,\alpha}$.
We begin by rewriting $I_1$:

\begin{align*}
I_1 &= \int_{\Omega_{1,2,\mu}}\int_\mathbb{R} \frac{(\phi_{1,2,\mu,\alpha}(x) - \phi_{1,2,\mu,\alpha}(y))^2}{|x-y|^{1+2\alpha}}dy - (\phi_{1,2,\mu,\alpha}(x))^2 dx \\ 
&= \int_{\Omega_1} \int_\mathbb{R} \frac{(\phi_{1,2,\mu,\alpha}(x) - \phi_{1,2,\mu,\alpha}(y))^2}{|x-y|^{1+2\alpha}}dy - (\phi_{1,2,\mu,\alpha}(x))^2 dx \\
& + \int_{\Omega_2} \int_\mathbb{R} \frac{(\phi_{1,2,\mu,\alpha}(x) - \phi_{1,2,\mu,\alpha}(y))^2}{|x-y|^{1+2\alpha}}dy - (\phi_{1,2,\mu,\alpha}(x))^2 dx \\
&=  \int_{\Omega_1} \int_{\Omega_1} \frac{(\phi_{1,2,\mu,\alpha}(x) - \phi_{1,2,\mu,\alpha}(y))^2}{|x-y|^{1+2\alpha}}dy - (\phi_{1,2,\mu,\alpha}(x))^2 dx   + \int_{\Omega_1} \int_{\mathbb{R} \backslash \Omega_{1,2,\mu}} \frac{(\phi_{1,2,\mu,\alpha}(x) )^2}{|x-y|^{1+2\alpha}}dy dx \\
&+ \int_{\Omega_2} \int_{\Omega_2} \frac{(\phi_{1,2,\mu,\alpha}(x) - \phi_{1,2,\mu,\alpha}(y))^2}{|x-y|^{1+2\alpha}}dy - (\phi_{1,2,\mu,\alpha}(x))^2 dx   + \int_{\Omega_2} \int_{\mathbb{R} \backslash \Omega_{1,2,\mu}} \frac{(\phi_{1,2,\mu,\alpha}(x) )^2}{|x-y|^{1+2\alpha}}dydx \\
&+2 \int_{\Omega_1} \int_{\Omega_2} \frac{(\phi_{1,2,\mu,\alpha}(x) - \phi_{1,2,\mu,\alpha}(y))^2}{|x-y|^{1+2\alpha}}dydx \\
&= \int_{\Omega_1} \int_{\Omega_1} \frac{(\phi_{1,2,\mu,\alpha}^1(x) - \phi_{1,2,\mu,\alpha}^1(y))^2}{|x-y|^{1+2\alpha}}dy - (\phi_{1,2,\mu,\alpha}^1(x))^2 dx   + \int_{\Omega_1} \int_{\mathbb{R} \backslash \Omega_1} \frac{(\phi_{1,2,\mu,\alpha}^1(x) )^2}{|x-y|^{1+2\alpha}}dydx \\
&+ \int_{\Omega_2} \int_{\Omega_2} \frac{(\phi_{1,2,\mu,\alpha}^2(x) - \phi_{1,2,\mu,\alpha}^2(y))^2}{|x-y|^{1+2\alpha}}dy - (\phi_{1,2,\mu,\alpha}^2(x))^2 dx   + \int_{\Omega_2} \int_{\mathbb{R} \backslash \Omega_2} \frac{(\phi_{1,2,\mu,\alpha}^2(x) )^2}{|x-y|^{1+2\alpha}}dy dx\\
&+2 \int_{\Omega_1} \int_{\Omega_2} \frac{(\phi_{1,2,\mu,\alpha}(x) - \phi_{1,2,\mu,\alpha}(y))^2}{|x-y|^{1+2\alpha}}dydx  \\
&- \int_{\Omega_1} \int_{\Omega_2}   \frac{(\phi_{1,2,\mu,\alpha}^1(x))^2}{|x-y|^{1+2\alpha}} dydx - \int_{\Omega_2} \int_{\Omega_1}   \frac{(\phi_{1,2,\mu,\alpha}^2(x))^2}{|x-y|^{1+2\alpha}} dydx.
\end{align*}
Finally, we find that 

\begin{equation}\label{pf_thm_2_i1}
\begin{aligned}
I_1& = \int_{\Omega_1} \int_{\Omega_1} \frac{(\phi_{1,2,\mu,\alpha}^1(x) - \phi_{1,2,\mu,\alpha}^1(y))^2}{|x-y|^{1+2\alpha}}dy - (\phi_{1,2,\mu,\alpha}^1(x))^2 dx   + \int_{\Omega_1} \int_{\mathbb{R} \backslash \Omega_1} \frac{(\phi_{1,2,\mu,\alpha}^1(x) )^2}{|x-y|^{1+2\alpha}}dydx \\
&+ \int_{\Omega_2} \int_{\Omega_2} \frac{(\phi_{1,2,\mu,\alpha}^2(x) - \phi_{1,2,\mu,\alpha}^2(y))^2}{|x-y|^{1+2\alpha}}dy - (\phi_{1,2,\mu,\alpha}^2(x))^2 dx   + \int_{\Omega_2} \int_{\mathbb{R} \backslash \Omega_2} \frac{(\phi_{1,2,\mu,\alpha}^2(x) )^2}{|x-y|^{1+2\alpha}}dydx \\
&+2 \int_{\Omega_1} \int_{\Omega_2} \frac{(\phi_{1,2,\mu,\alpha}(x) - \phi_{1,2,\mu,\alpha}(y))^2}{|x-y|^{1+2\alpha}}dydx  \\
&- \int_{\Omega_2} \int_{\Omega_1}   \frac{(\phi_{1,2,\mu,\alpha}^2(x))^2}{|x-y|^{1+2\alpha}} dydx - \int_{\Omega_1} \int_{\Omega_2}   \frac{(\phi_{1,2,\mu,\alpha}^1(x))^2}{|x-y|^{1+2\alpha}} dydx.
\end{aligned}
\end{equation}
With similar computations, we find that for $I_2$:
%\begin{align*}
%I_2 &= \int_{\mathbb{R} \backslash \Omega}\int_\mathbb{R} \frac{(\phi_{1,2,\mu,\alpha}(x) - \phi_{1,2,\mu,\alpha}(y))^2}{|x-y|^{1+2\alpha}}dydx \\
%&= \int_{\mathbb{R} \backslash \Omega}\int_{\mathbb{R}} \frac{(\phi_{1,2,\mu,\alpha}(y))^2}{|x-y|^{1+2\alpha}}dydx \\
%&=  \int_{\mathbb{R} \backslash \Omega}\int_{\Omega} \frac{(\phi_{1,2,\mu,\alpha}(y))^2}{|x-y|^{1+2\alpha}}dydx \\
%&= \int_{\Omega_1} \int_{\mathbb{R} \backslash \Omega}  \frac{(\phi_{1,2,\mu,\alpha}(y))^2}{|x-y|^{1+2\alpha}}dxdy +  \int_{\Omega_2} \int_{\mathbb{R} \backslash \Omega}  \frac{(\phi_{1,2,\mu,\alpha}(y))^2}{|x-y|^{1+2\alpha}}dxdy \\
%&=  \int_{\Omega_1} \int_{\mathbb{R} \backslash \Omega_1}  \frac{(\phi_{1,2,\mu,\alpha}^1(y))^2}{|x-y|^{1+2\alpha}}dxdy +\int_{\Omega_2} \int_{\mathbb{R} \backslash \Omega_2}  \frac{(\phi_{1,2,\mu,\alpha}^2(y))^2}{|x-y|^{1+2\alpha}}dxdy \\
%& -  \int_{\Omega_1} \int_{\Omega_2} \frac{(\phi_{1,2,\mu,\alpha}^1(y))^2}{|x-y|^{1+2\alpha}}dxdy  - \int_{\Omega_2} \int_{\Omega_1} \frac{(\phi_{1,2,\mu,\alpha}^2(y))^2}{|x-y|^{1+2\alpha}}dxdy .
%\end{align*}
%Finally, we find that
\begin{equation}\label{pf_thm_2_i2}
\begin{aligned}
I_2 &=  \int_{\Omega_1} \int_{\mathbb{R} \backslash \Omega_1}  \frac{(\phi_{1,2,\mu,\alpha}^1(y))^2}{|x-y|^{1+2\alpha}}dxdy +\int_{\Omega_2} \int_{\mathbb{R} \backslash \Omega_2}  \frac{(\phi_{1,2,\mu,\alpha}^2(y))^2}{|x-y|^{1+2\alpha}}dxdy \\
& -  \int_{\Omega_1} \int_{\Omega_2} \frac{(\phi_{1,2,\mu,\alpha}^1(y))^2}{|x-y|^{1+2\alpha}}dxdy  - \int_{\Omega_2} \int_{\Omega_1} \frac{(\phi_{1,2,\mu,\alpha}^2(y))^2}{|x-y|^{1+2\alpha}}dxdy .
\end{aligned}
\end{equation}
Combining \eqref{pf_thm_2_i1} and \eqref{pf_thm_2_i2} in \eqref{pf_thm_2_lambda}, we deduce 
\begin{align*}
&\lambda_\alpha (\Omega_{1,2,\mu }) = \frac{2  \int_{\Omega_1} \int_{\Omega_2} \frac{(\phi_{1,2,\mu,\alpha}(x) -\phi_{1,2,\mu,\alpha}(y))^2}{|x-y|^{1+2\alpha}}dydx  }{\int_{\Omega_{1,2,\mu}} (\phi_{1,2,\mu,\alpha}(x))^2 dx} \\
&+ \frac{\int_{\Omega_1} (\phi_{1,2,\mu,\alpha}^1(x))^2 dx   \left(\frac{\int_{\Omega_1} \int_{\Omega_1} \frac{(\phi_{1,2,\mu,\alpha}^1(x) - \phi_{1,2,\mu,\alpha}^1(y))^2}{|x-y|^{1+2\alpha}}dy - (\phi_{1,2,\mu,\alpha}^1(x))^2 dx   + 2  \int_{\Omega_1} \int_{\mathbb{R} \backslash \Omega_1} \frac{(\phi_{1,2,\mu,\alpha}^1(x) )^2}{|x-y|^{1+2\alpha}}dydx}{\int_{\Omega_1} (\phi_{1,2,\mu,\alpha}^1(x))^2 dx } \right)}{\int_{\Omega_{1,2,\mu}} (\phi_{1,2,\mu,\alpha}(x))^2 dx}\\
&+ \frac{\int_{\Omega_2} (\phi_{1,2,\mu,\alpha}^2(x))^2 dx   \left(\frac{\int_{\Omega_2} \int_{\Omega_2} \frac{(\phi_{1,2,\mu,\alpha}^2(x) - \phi_{1,2,\mu}^2(y))^2}{|x-y|^{1+2\alpha}}dy - (\phi_{1,2,\mu,\alpha}^2(x))^2 dx   + 2  \int_{\Omega_2} \int_{\mathbb{R} \backslash \Omega_2} \frac{(\phi_{1,2,\mu,\alpha}^2(x) )^2}{|x-y|^{1+2\alpha}}dydx}{\int_{\Omega_2} (\phi_{1,2,\mu,\alpha}^2(x))^2 dx } \right)}{\int_{\Omega} (\phi_{1,2,\mu,\alpha}(x))^2 dx} \\ 
&-\frac{  2 \int_{\Omega_1} \int_{\Omega_2}   \frac{(\phi_{1,2,\mu,\alpha}^1(x))^2}{|x-y|^{1+2\alpha}} dydx + 2\int_{\Omega_2} \int_{\Omega_1}   \frac{(\phi_{1,2,\mu,\alpha}^2(x))^2}{|x-y|^{1+2\alpha}} dydx }{\int_{\Omega_{1,2,\mu}} (\phi_{1,2,\mu,\alpha}(x))^2 dx} \\
&\geq \frac{\int_{\Omega_1} (\phi_{1,2,\mu,\alpha}^1(x))^2 dx }{\int_{\Omega_{1,2,\mu}} (\phi_{1,2,\mu,\alpha}(x))^2 dx}  \underline{\lambda}_{1,\alpha} + \frac{\int_{\Omega_2} (\phi_{1,2,\mu,\alpha}^2(x))^2 dx }{\int_{\Omega_{1,2,\mu}} (\phi_{1,2,\mu,\alpha}(x))^2 dx} \underline{\lambda}_{2,\alpha} \\
\ & \ -\frac{ 2\int_{\Omega_1} \int_{\Omega_2}   \frac{(\phi_{1,2,\mu,\alpha}^1(x))^2}{|x-y|^{1+2\alpha}} dydx + 2\int_{\Omega_2} \int_{\Omega_1}   \frac{(\phi_{1,2,\mu,\alpha}^2(x))^2}{|x-y|^{1+2\alpha}} dydx}{\int_{\Omega_{1,2,\mu}} (\phi_{1,2,\mu,\alpha}(x))^2 dx} \\
&\geq \min(\underline{\lambda}_{1,\alpha}, \underline{\lambda}_{2,\alpha}) -\frac{ 2\int_{\Omega_1} \int_{\Omega_2}   \frac{(\phi_{1,2,\mu,\alpha}^1(x))^2}{|x-y|^{1+2\alpha}} dydx + 2\int_{\Omega_2} \int_{\Omega_1}   \frac{(\phi_{1,2,\mu,\alpha}^2(x))^2}{|x-y|^{1+2\alpha}} dydx}{\int_{\Omega_{1,2,\mu}} (\phi_{1,2,\mu,\alpha}(x))^2 dx} .
\end{align*}
But, for $i \in \left\lbrace 1,2 \right\rbrace$, we have for all $x \in \Omega_i$ and $ y \in \Omega_{3-i}$ 
\[2\mu  < |x-y| \Rightarrow  -\frac{1}{(2\mu)^{1+2\alpha}} \leq -\frac{1}{|x-y|^{1+2\alpha}}.\]
We deduce that
\begin{align*}
\lambda_\alpha (\Omega_{1,2,\mu }) &\geq \min(\underline{\lambda}_{1,\alpha}, \underline{\lambda}_{2,\alpha}) -\frac{ 2\int_{\Omega_1} \int_{\Omega_2}   \frac{(\phi_{1,2,\mu,\alpha}^1(x))^2}{|x-y|^{1+2\alpha}} dydx + 2\int_{\Omega_2} \int_{\Omega_1}   \frac{(\phi_{1,2,\mu,\alpha}^2(x))^2}{|x-y|^{1+2\alpha}} dydx}{\int_{\Omega_{1,2,\mu}} (\phi_{1,2,\mu,\alpha}(x))^2 dx} \\
&\geq \min(\underline{\lambda}_{1,\alpha}, \underline{\lambda}_{2,\alpha}) -\frac{ 2\int_{\Omega_1} \int_{\Omega_2}   \frac{(\phi_{1,2,\mu,\alpha}^1(x))^2}{\mu^{1+2\alpha}} dydx + 2\int_{\Omega_2} \int_{\Omega_1}   \frac{(\phi_{1,2,\mu,\alpha}^2(x))^2}{\mu^{1+2\alpha}} dydx}{\int_{\Omega_{1,2,\mu}} (\phi_{1,2,\mu,\alpha}(x))^2 dx} \\
& \geq \min(\underline{\lambda}_{1,\alpha}, \underline{\lambda}_{2,\alpha}) - \frac{4A}{(2\mu)^{1+2\alpha}} \underset{\mu \rightarrow + \infty}\longrightarrow \min(\underline{\lambda}_{1,\alpha}, \underline{\lambda}_{2,\alpha}). 
\end{align*}
Since $\min(\underline{\lambda}_{1,\alpha}, \underline{\lambda}_{2,\alpha})>0$, we deduce the existence of $\mu_1>0$ such that for all $\mu > \mu_1$, 
\[\lambda_\alpha (\Omega_{1,2,\mu }) > 0.\]
\end{proof}

\bibliographystyle{plain} % Le style est mis entre accolades.
%\bibliography{/home/aleculie/Documents/Bibliographie/bibliographie.bib} % mon fichier de base de donn�es s'appelle bibli.bib
{\footnotesize
%\bibliography{/home/aleculie/Documents/Austin/Documents/Bibliographie/bibliographie.bib}}
\bibliography{Biblio}}
%\bibliography{/Desktop/Austin/Documents/Bibliographie/bibliographie.bib}}
%bibtex %.aux
	
%\end{thebibliography}

\end{document}